\documentclass[runningheads]{llncs}
\usepackage{graphicx}
\usepackage{tikz} 
\usepackage{amsmath}
\usepackage{bbm}
\usepackage{url}
\usepackage{algorithm}
\usepackage{algpseudocode}
\usepackage{multirow}
\usetikzlibrary{decorations.pathreplacing}
\DeclareMathOperator*{\argmin}{argmin}
\DeclareMathOperator*{\argmax}{argmax}
\newcommand{\parent}{\mathrm{par}}
\newcommand{\child}{\mathrm{child}}

\newcommand{\supp}{\mathrm{supp}}

\newcommand{\range}{\mathrm{Image}}

\newcommand{\ignore}[1]{}

\usepackage[normalem]{ulem}
\newtheorem{Claim}{Claim}

\begin{document}
	\title{A Parametric Approach for Solving Convex Quadratic Optimization with Indicators Over Trees\thanks{This research is supported, in part, by  NSF grants 2006762, 2007814, 2152776, 2152777, 2337776, ONR grant N00014-22-1-2127 and AFOSR grant FA9550-22-1-0369.}}
	\author{Aaresh Bhathena\inst{1} \and
		Salar Fattahi\inst{1} \and
		Andr\'es G\'omez\inst{2}\and Simge K{\"u}{\c{c}}{\"u}kyavuz \inst{3}}
	
 \titlerunning{Quadratic Optimization with Indicators Over Trees}
	\authorrunning{A.\ Bhathena et al.}
	
	\institute{University of Michigan, Ann Arbor, MI, USA. \email{\{aareshfb,fattahi\}@umich.edu} \and
		University of Southern California, Los Angeles, CA, USA.
		\email{gomezand@usc.edu}\and
		Northwestern University,
		Evanston, IL, USA.
		\email{simge@northwestern.edu}}
	
	\maketitle             
	
	\begin{abstract}
		This paper investigates convex quadratic optimization problems involving $n$ indicator variables, each associated with a continuous variable, particularly focusing on scenarios where the matrix $Q$ defining the quadratic term is positive definite and its sparsity pattern corresponds to the adjacency matrix of a tree graph. We introduce a graph-based dynamic programming algorithm that solves this problem in time and memory complexity of $\mathcal{O}(n^2)$. Central to our algorithm is a precise parametric characterization of the cost function across various nodes of the graph corresponding to distinct variables. Our computational experiments conducted on both synthetic and real-world datasets demonstrate the superior performance of our proposed algorithm compared to existing algorithms and state-of-the-art mixed-integer optimization solvers. An important application of our algorithm is in the real-time inference of Gaussian hidden Markov models from data affected by outlier noise. Using a real on-body accelerometer dataset, we solve instances of this problem with over 30,000 variables in under a minute, and its online variant within milliseconds on a standard computer. A Python implementation of our algorithm is available at \url{https://github.com/aareshfb/Tree-Parametric-Algorithm.git}.
	\end{abstract}
	\section{Introduction}
	Given a symmetric and positive definite matrix $Q\in\bbbr^{n\times n}$ and vectors $\lambda,c\in\bbbr^n$, we study the following mixed-integer quadratic optimization (MIQP) problem:
	\begin{subequations}\label{eq: MIQP}
		\begin{align}
			\min_{x\in\bbbr^n,z\in\{0,1\}^n}\qquad& \dfrac{1}{2}x^\top Qx+c^\top x+\lambda^\top z\label{eq: MIQP obj}\\ 
			\text{s.t.}\qquad &x_i(1-z_i)=0&  i=1,2,\ldots, n\label{eq: MIQP_const}
		\end{align}
	\end{subequations}
	Specifically, we assume that the sparsity pattern of $Q\in \bbbr^{n\times n}$ is the adjacency matrix of a connected tree.
	The binary vector $z\in \{0,1\}^n$ is used to model the support of the vector $x\in \bbbr^n$ and $\lambda\in \bbbr^n$ is a regularization parameter on the sparsity of $x$. If $z_i=0$, then from constraint \eqref{eq: MIQP_const}, we obtain $x_i=0$. On the other hand, if $z_i=1$ then $x_i$ is unconstrained. Without loss of generality, we assume that the diagonal elements of $ Q $ are equal to $ 1 $. This can be ensured by replacing $x_i\leftarrow x_i/\sqrt{Q_{i,i}}$ for all $i=1,\dots,n$. We also assume $\lambda_i>0$ for every $i\in \{1,2,\dots, n\}$, as any $\lambda_i\le 0$ implies $z_i = 1$ at optimality. The regularizer $\lambda$ can model the sparsity of the solution, which makes the above problem useful in network inference \cite{gomez2021outlier,visweswaran2023efficient}, sparse regression \cite{bertsimas2016subset,del2020subset}, and probabilistic graphical models \cite{kucukyavuz2022consistent,liu2023graph,Manzour21}, to name a few.
	
	\subsection{Gaussian hidden Markov models}\label{sec: HMM}
	An important application of Problem \eqref{eq: MIQP} is in the inference of {\it Gaussian hidden Markov models} (GHMM) \cite{atamturk2021sparse,gomez2021outlier},
	 where the goal is to estimate hidden states $\{x_t\}_{t=1}^T$ of a random process from $K_t$ observations $\{y_{k,t}\}_{k=1}^{K_t}$ at each time $t$. We consider the Besag model~\cite{besag1974spatial,besag1995conditional}, where the hidden states are assumed to be jointly Gaussian:
	\begin{align}
		p(x_1,\dots,x_T)\propto \exp\left(-\frac{1}{2\sigma_1}x_1^2-\sum_{i=2}^{T}\frac{1}{2\sigma_t^2}\left(x_t-x_{t-1}\right)^2\right).
	\end{align}
	Each hidden state $x_t$ is indirectly observed via noisy observations $y_{k,t} = x_t+\epsilon_{k,t}+\delta_{k,t}, k=1,\dots, K_t$, where $\epsilon_{k,t}$ is a dense, but light-tailed noise drawn from $\mathcal{N}(0,\nu_t^2)$, whereas $\delta_{k,t}$ is an outlier noise that corrupts only a small subset of the observations. An example of a GHMM is given in Figure \ref{fig: Hidden Markov Model}.
	
	\begin{figure}
		\begin{tikzpicture}
			[scale=.78,auto=center,
			every node/.style={circle,fill=blue!20, draw=blue!60,very thick,minimum size=7mm},
			bluenode/.style={circle,fill=blue!20, draw=blue!60,very thick,minimum size=7mm},
			rednode/.style={circle, draw=red!60, fill=red!5, very thick, minimum size=7mm},
			yellownode/.style={circle, draw=yellow!100, fill=yellow!15, very thick, minimum size=7mm},
			greennode/.style={circle, draw=green!60, fill=green!15, very thick, minimum size=7mm}]   
			\node[bluenode] (t1) at (0,0)  {$x_1$}; 
			\node[bluenode] (t2) at (3,0)  {$x_2$};
			\node[bluenode] (t3) at (6,0)  {$x_3$};
			\node[bluenode] (t4) at (11,0)  {$x_T$};
			
			\node[yellownode] (y11) at (-1,2) {\small $y_{1,1}$};
			\node[rednode] (y12) at (-2,0.5) {\small $y_{2,1}$};
			\node[yellownode] (y13) at (-1,-2) {\small $y_{3,1}$};	
			
			\node[yellownode] (y21) at (3,-2) {\small $y_{1,2}$};
			\node[yellownode] (y31) at (5,2) {\small $y_{1,3}$};
			\node[rednode] (y32) at (7,2) {\small $y_{2,3}$};	
			
			\node[yellownode] (yn4) at (12.75,1) {\small $y_{4,T}$};
			\node[rednode] (yn3) at (10,-2) {\small $y_{3,T}$};
			\node[yellownode] (yn2) at (12,-2) {\small $y_{2,T}$};
			\node[yellownode] (yn1) at (11,2) {\small $y_{1,T}$};	
			
			\filldraw[black] (8.5,0) circle (1.5pt);
			\filldraw[black] (8.7,0) circle (1.5pt);
			\filldraw[black] (8.9,0) circle (1.5pt);
			\filldraw[black] (9.1,0) circle (1.5pt);

			\draw[stealth-] (t2)--(t1);
			\draw[stealth-] (t3)--(t2);
			\draw[-stealth] (t3)--(8,0);
			\draw[stealth-] (t4)--(9.5,0);
			
			\draw[stealth-] (t1)--(y11);
			\draw[stealth-] (t1)--(y12);
			\draw[stealth-] (t1)--(y13);
			\draw[stealth-] (t2)--(y21);
			\draw[stealth-] (t3)--(y31);
			\draw[stealth-] (t3)--(y32);
			\draw[stealth-] (t4)--(yn1);
			\draw[stealth-] (t4)--(yn2);
			\draw[stealth-] (t4)--(yn3);
			\draw[stealth-] (t4)--(yn4);
		\end{tikzpicture} 
		\caption{An illustration of a GHMM. At every time $t$, observations $\{y_{k,t}\}_{k=1}^{K_t}$ of the hidden state $x_t$ are collected, some of which may be corrupted with outlier noise (shown in red). The goal is to infer the hidden states $\{x_t\}_{t=1}^T$ from these corrupted observations.}
		\label{fig: Hidden Markov Model}
	\end{figure}
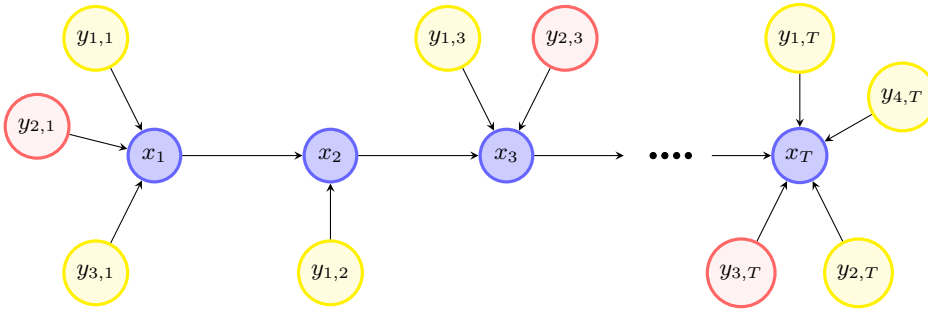
	
	One of the earliest applications of GHMMs can be traced back to signal processing, aimed at predicting the evolution of a random signal over time by effectively filtering out observational noise~\cite{brown1992introduction,kalman1960new}. A more contemporary application of GHMMs lies in activity recognition utilizing on-body wearable accelerometers~\cite{kim2015hidden,trabelsi2013unsupervised}. In such contexts, additional consideration may involve assuming sparsity in the underlying hidden state $\{x_t\}_{t=1}^T$, which corresponds to the inactive state of the body. Under such settings, it is natural to consider the \textit{maximum a posteriori estimate} of the hidden states with $\ell_0$ regularization to promote the sparsity prior on the outliers as well as the underlying hidden states. This problem can be formulated as follows:
	\begin{subequations}\label{eq: HMM}
		\begin{align}
			\min_{x,z,w,s} \quad& \sum_{t=1}^{T}\sum_{k=1}^{K_t}\frac{1}{\nu_t^2}\left(y_{k,t}\!-\!x_t\!-\!w_{k,t}\right)^2+\frac{1}{\sigma_1^2}x_1^2\!+\!\sum_{t=2}^{T}\frac{1}{\sigma_t^2}(x_{t}\!-\!x_{t-1})^2\nonumber\\
			&\!+\!\sum_{t=1}^T\sum_{k=1}^{K_t}\lambda_{k,t}z_{k,t}\!+\!\sum_{t=1}^T\gamma_t s_t\\
			\text{s.t.}\quad
			&w_{k,t}(1-z_{t,k})=0\qquad\qquad\qquad\enspace   t=1,2\ldots T; k=1,\ldots K_t\\
			&x_t(1-s_t)=0\qquad\qquad\qquad\qquad t=1,2\ldots T\\
			&w_{\cdot, t}\in\bbbr^{K_t},z_{\cdot,t}\in\{0,1\}^{K_t}\qquad\quad t=1,2\ldots T\\
			&x\in \bbbr^T,s\in\{0,1\}^T.
		\end{align}
	\end{subequations}
	In the optimization problem \eqref{eq: HMM}, the binary variables $\{z_{k,t}\}$ capture the presence of outlier noise in observations $\{y_{k,t}\}$. Specifically, $z_{k,t}=1$ indicates that $y_{k,t}$ is likely to be tainted with noise. This can be understood by noting that when $z_{k,t}=1$, the continuous variable $w_{k,t}$ takes the value $y_{k,t}-x_t$ at optimality, thereby nullifying the impact of the observation $y_{k,t}$ on the estimated state $x_t$. Conversely, $z_{k,t}=0$ implies $w_{k,t}=0$, indicating that the observation $y_{k,t}$ is devoid of outlier noise. Moreover, the binary variables $\{s_t\}$ capture the support of the hidden state $\{x_t\}$, enforcing $x_t=0$ if and only if $s_t=0$. 	
	The above optimization problem is a special case of Problem~\eqref{eq: MIQP}, where the matrix $Q$ is positive definite and its support is the adjacency matrix of a tree graph (as can be seen in Figure~\ref{fig: Hidden Markov Model}). 
{An important variant of problem \eqref{eq: HMM}, arising in real-time monitoring and detection of events, is the online variant where observations $\{y_{k,t}\}$ become available over time \cite{yan2022real}. In such scenarios, where rapid action is necessary upon detecting anomalous events, re-optimization of Problem \eqref{eq: HMM} must be performed within milliseconds. 

In most cases, Problem \eqref{eq: HMM} is rarely tackled in the literature directly. Indeed, mixed-integer nonlinear optimization problems are typically regarded as intractable. Moreover, big-M relaxations of \eqref{eq: HMM} result in poor relaxations with trivial lower bounds, thus simply resorting to off-the-shelf solvers may prove ineffective. Thus, practitioners often resort to simpler approximations, consisting of either using $\ell_1$-norm penalty to induce sparsity on variables $x$ and $w$ \cite{yan2022real}, or using iterative procedures and heuristics to remove outliers \cite{chang1988estimation,tsay1986time}. Naturally, such approximations admit fast algorithms, but the solution quality can be negatively affected. Recently, there has been a renewed interest in developing improved mixed-integer optimization formulations for problems with sparsity and outliers \cite{gomez2021outlier,gomez2023outlier,han2021compact,insolia2022simultaneous}. The results indicate that exact methods can indeed deliver substantially better solutions, especially when outliers are clustered. Typical runtimes of mixed-integer optimization methods with strong formulations is measured in minutes for problems with $T$ in the hundreds, which is adequate for small-sized offline versions for \eqref{eq: HMM}, but far from practical for online problems.} In this paper, we propose a method that solves the online problem to optimality within milliseconds on a standard computer.

	\subsection{Related work}
	For general dense matrix $Q$, Problem \eqref{eq: MIQP} is NP-hard \cite{chen2014NPhard,huo2010NPhard}. Earlier methods based on mixed-integer programming using big-M formulation \cite{bertsimas2016subset,bertsimas2020sparse,dedieu2021learning} work reliably for small instances, but exhibit poor scalability \cite{hastie2017extended,hazimeh2022sparse}. Since then, there has been a significant improvement in solving these problems over large instances. One key contribution was the \textit{perspective-reformulation} technique that obtains high-quality convex relaxations of the feasible region. 
Initially introduced in \cite{stubbs1996branch}, perspective reformulations have served as the cornerstone for numerous methods aimed at solving Problem~\eqref{eq: MIQP} with general $Q$, either exactly or approximately \cite{bertsimas2023new,gomez2024note,gunluk2010perspective,gunluk2011perspective,Wei2024,wei2021ideal,wei2020convexification}. 
	
Due to the NP-hardness of Problem~\eqref{eq: MIQP} with a general $Q$, recent endeavors have shifted focus towards cases where $Q$ possesses a special structure. When $Q$ exhibits a diagonal structure, it has been demonstrated that a perspective reformulation already yields the ideal convex hull characterization of Problem~\eqref{eq: MIQP} \cite{ceria1999convex}. Moreover, if the matrix $Q$ can be factorized as $Q=Q_0^\top Q_0$, where $Q_0$ is sparse, the problem can be solved in polynomial time under appropriate conditions \cite{del2020subset}. In \cite{das2008algorithms}, a cardinality-constrained version of Problem \eqref{eq: MIQP} is explored, where $Q$ corresponds to a tree with a maximum degree $d$, and all coefficients $\lambda_i$ are identical. The authors propose a dynamic programming algorithm that operates in $\mathcal{O}(n^3d)$ time. Consequently, this leads to a $\mathcal{O}(n^4)$ algorithm for the regularized version discussed in this paper, with the additional restriction that all coefficients $\lambda_i$ are identical.

Perhaps most closely related to our work are two lines of research that investigate Problem~\eqref{eq: MIQP} when $Q$ possesses either a path or Stieltjes structure. 
When $Q$ is a Stieltjes matrix, it is recently shown that Problem~\eqref{eq: MIQP} can be converted into a submodular minimization problem and thus solved in polynomial time \cite{atamturk2018strong,han2022polynomial}. Any $Q$ that has a tree structure can be turned into a Stieltjes matrix with a simple change of variables. Therefore, Problem~\eqref{eq: MIQP} can be solved in polynomial time. An application of the state-of-the-art submodular minimization algorithm leads to a runtime of $\mathcal{O}(n^5 \texttt{EO})$, where $\texttt{EO}$ is the complexity of solving a certain quadratic program~\cite{orlin2009faster}. Although this complexity is expected to be improved with more recent algorithms such as those introduced in \cite{chakrabarty2017subquadratic,lee2015faster}, they remain inefficient in medium to large-scale instances.
When $Q$ takes the form of a tridiagonal matrix, the works \cite{fattahi2021scalable,fattahi2023solution,liu2023graph} have introduced dynamic programming (DP) algorithms capable of solving Problem~\eqref{eq: MIQP} in $\mathcal{O}(n^2)$. However, in Section~\ref{sec: drawback of DP}, we provide a detailed discussion on why these dynamic programming algorithms cannot be readily extended to the more general case of tree structures for $Q$.

 \subsection{Preliminaries and notations}
	Given a matrix $Q\in \bbbr^{n\times n}$ and index sets $I$ and $J$, we denote by $Q_{I,J}$ the sub-matrix of $Q$ whose rows and columns correspond to $I$ and $J$, respectively. Similarly, given a vector $c\in \bbbr^n$, we denote by $c_J$ a sub-vector of $c$ with indices corresponding to $J$. For integers $k<l$, we define $[k:l]=[k, k+1, \dots, l]$. We use $\bbbone_{x}$ to denote the indicator function defined over $\bbbr$ that returns $0$ for $x= 0$ and returns $1$ for all $x\ne 0$.
	The notations $f^\star$ and $(x^\star, z^\star)$ are used to denote the optimal objective value and optimal solution of Problem \eqref{eq: MIQP} respectively.
	\begin{definition}
		Given a symmetric matrix $Q\in \bbbr^{n\times n}$, the \textbf{support graph} of $Q$, denoted by $\operatorname{supp}(Q)$, is an undirected graph $\mathcal{G}=(N,E)$, where $N=\{1,\ldots,n\}$ and $(i,j)\in E$ if and only if $Q_{i,j}\ne 0$ for $i\ne j$.
	\end{definition}

	In this paper, we consider problems where $\mathrm{supp}(Q)$ is a tree. Without loss of generality, we assume that $\mathrm{supp}(Q)$ is connected and rooted. Moreover, we assume that the edges have a natural orientation toward the root node.\footnote{Recall that $\supp(Q)$ is undirected; the natural orientation assumption is only to streamline the presentation of our algorithm.} We use $\mathrm{child}(u)$ to denote the child node of $u$ in the rooted tree $\supp(Q)$. Similarly, we use $\parent(u)$ to denote the set of parent nodes of $u$.  We assume a topological ordering for the nodes in $\supp(Q)$. More specifically, we assume $u<\child(u)$ for every node in $\supp(Q)$. Therefore, node $n$ is always the root node. Since $\supp(Q)$ is a tree, its topological labeling always exists and can be obtained in $\mathcal{O}(n)$ time and memory~\cite[Algorithm 3.8]{ahuja1988network}. Moreover, due to the considered directions, each node can only have a single child, but potentially multiple parents. Figure~\ref{fig: Label for trees} illustrates the topological ordering of an exemplary tree. The degree of each node $u$ in $\mathrm{supp}(Q)$ is denoted as $\mathrm{deg}(u)$. If $\mathrm{deg}(u)\geq 3$, we say $u$ is a \textit{branch}. Trees with only one branch are referred to as \textit{extended star trees}. 
 \begin{figure}
     \centering
     \begin{tikzpicture}  
			[scale=.7,auto=center,every node/.style={circle,draw=black!100, very thick, minimum size=7mm},roundnode/.style={circle, draw=red!60, fill=red!5, very thick, minimum size=7mm},yellownode/.style={circle, draw=yellow!100, fill=yellow!15, very thick, minimum size=7mm},greennode/.style={circle, draw=green!60, fill=green!15, very thick, minimum size=7mm},bluenode/.style={circle, draw=blue!60, fill=blue!20, very thick, minimum size=7mm}] 

            
			\node (a1) at (2.5,0) {12};  
			\node (a2) at (-1,-2)  {8}; 
			\node (a17) at (-2.5,-4)  {7}; 
			\node (a18) at (-2.5,-6)  {6}; 

			\node (a9) at (1,-4)  {5};
			\node (a10) at (1,-6)  {4};
			\node (a11) at (-0.5,-8)  {2};
			\node (a12) at (-0.5,-10)  {1};
			\node (a13) at (2.5,-8)  {3};
			\node (a14) at (6,-2)  {11};
			\node (a15) at (4.5,-4)  {9};
			\node (a16) at (7.5,-4)  {10};
			\draw[<-, line width=0.5mm] (a1) -- (a2); 
 
			\draw[<-, line width=0.5mm] (a2) -- (a9); 
			\draw[<-, line width=0.5mm] (a9) -- (a10); 
			\draw[<-, line width=0.5mm] (a10) -- (a11); 
			\draw[<-, line width=0.5mm] (a11) -- (a12);
			\draw[<-, line width=0.5mm] (a10) -- (a13); 
			\draw[<-, line width=0.5mm] (a1) -- (a14);
			
			\draw[<-, line width=0.5mm] (a2) -- (a17);
			\draw[<-, line width=0.5mm] (a17) -- (a18);
			\draw[<-, line width=0.5mm] (a14) -- (a15);
			\draw[stealth-, line width=0.5mm] (a14) -- (a16);
		\end{tikzpicture}
     \caption{An example of the topological labeling of nodes of a tree. In this example, $\child(4)=5$ and $\parent(4)=\{2,3\}$.}
     \label{fig: Label for trees}
 \end{figure}
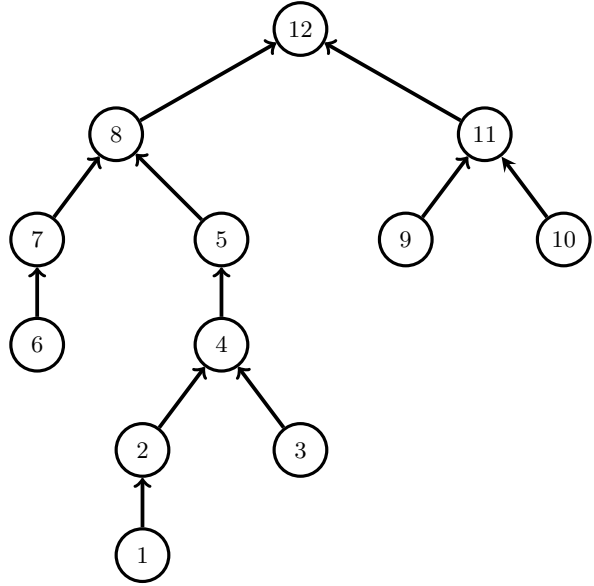
 
 Given any node $u$, $\mathrm{supp}_u(Q)$ denotes the largest connected sub-tree of $\mathrm{supp}(Q)$ comprised of $u$ and its descendants, that is, any node $v\leq u$ that has a path to $u$. The symbol $n_u$ denotes the number of nodes in $\mathrm{supp}_u(Q)$. Given a node $u$, we define $Q_{[u]}$ as the sub-matrix of $Q$ with rows and columns corresponding to the nodes in $\supp_u(Q)$. It follows that $\supp_u(Q) = \supp(Q_{[u]})$. Similarly, we define $c_{[u]}$ and $\lambda_{[u]}$ as the sub-vectors of $c$ and $\lambda$ with indices corresponding to the nodes of $\supp(Q_{[u]})$.
 
 We refer to $f_u(\alpha)$ as the \textit{parametric cost at node $u$}, defined as:
 \begin{subequations}\label{eq: MIQP_u}
		\begin{align}
			f_u(\alpha):=\min_{x\in\bbbr^{n_u},z\in\{0,1\}^{n_u}}\qquad& \dfrac{1}{2}x^\top Q_{[u]}x+{c^\top_{[u]} x}+{\lambda^\top_{[u]}} z\label{eq: MIQP obj_u}\\ 
			\text{s.t.}\qquad &x_i(1-z_i)=0 \qquad \qquad  i=1,2\ldots, n_u\label{eq: MIQP_const_u}\\
			& x_{n_u}=\alpha.&
		\end{align}
	\end{subequations}
	In simpler terms, $f_u(\alpha)$ represents the optimal cost of the sub-problem defined over the sub-tree $\supp_u(Q)$ when the root node variable $x_{n_u}$ is set to $\alpha$. 
	We define $f_u^\star = \min_{\alpha} f_u(\alpha)$. 
	Indeed, we have $f^* = f^\star_n$. 
	
 Our next lemma shows that $f_u(\alpha)$ can be written as the sum of an indicator function and the minimum of at most $2^{n_u-1}$ quadratic and strongly convex functions, $p_{u,s}(\cdot)$, where $s$ is a binary vector of dimension $n_u-1$ that indicates which variables in the subtree are nonzero. For each configuration of $s$, the resulting optimization problem can be shown to be strongly convex quadratic. The proof of this lemma is presented in Appendix~\ref{lem:f_quad}. 
	\begin{lemma}\label{lem:f_quad}
		Suppose that $Q$ is positive definite and $\supp(Q)$ is a tree. For any $1\leq u\leq n$, $f_u(\alpha)$ can be written as:
		$$
		f_u(\alpha) = \min_{s\in\{0,1\}^{n_u-1}}\{p_{u,s}(\alpha)\}+\lambda_u\bbbone_{\alpha},
		$$
		where, for every $s\in\{0,1\}^{n_u-1}$, $p_{u,s}(\alpha)$ is quadratic and strongly convex.
	\end{lemma}

This paper extensively uses conjugate functions. Recall that, given a function $f:\bbbr\to\bbbr$, its Fenchel conjugate is defined as $$f^*(\beta) = \sup_{\alpha}\{\alpha\beta-f(\alpha)\}.$$ A fundamental property of $f^*(\beta)$ is that it is convex, even if $f(\alpha)$ is not. Moreover, $f(\alpha)+f^*(\beta)\geq \beta\alpha$, for every $\alpha,\beta\in\bbbr$.

	\section{Dynamic programming over trees} \label{sec: drawback of DP}
	
	When $\supp(Q)$ is a path graph, Problem~\eqref{eq: MIQP} reduces to the following optimization problem:
	\begin{subequations}\label{eq: MIQP_path}
		\begin{align}
			\min_{x\in\bbbr^n,z\in\{0,1\}^n}\qquad& \left(\dfrac{1}{2}\sum_{i=1}^n x_i^2 +\sum_{i=2}^n Q_{i,i-1}x_ix_{i-1}\right)+c^\top x+\lambda^\top z\label{eq: MIQP obj_path}\\ 
			\text{s.t.}\qquad &x_i(1-z_i)=0\qquad i=1,2\ldots, n.\label{eq: MIQP_const_path}
		\end{align}
	\end{subequations}
	Liu et al.~\cite{liu2023graph} introduced a DP approach for solving the above problem.  To explain this method, let $q^\star_{[k:l]}$ denote the optimal cost of~\eqref{eq: MIQP_path} with additional constraints $z_i=1$ for every $k\leq i\leq l$ and $z_i=0$ otherwise. A simple calculation reveals that
	\begin{align*}
		\quad q^\star_{[k:l]} = \begin{cases}
			-\dfrac{1}{2}c_{[k:l]}^\top \left(Q_{[k:l], [k:l]}\right)^{-1}c_{[k:l]}+\sum_{i=k}^l\lambda_i\quad &1\leq k\leq l\leq n\\
			0 & k>l.
		\end{cases}
	\end{align*}
	Let $s$ be the largest index such that $z^\star_s=0$. 
	Upon setting $z_s=0$ in \eqref{eq: MIQP_path}, the problem decomposes into two sub-problems: one defined over nodes $\{1, \dots, s-1\}$ and the other over nodes $\{s+1,\dots, n\}$ with the additional constraint that $z_i=1$ for every $i>s$. This decomposition implies that $f^\star = q^\star_{[s+1:n]}+f_{s-1}^{\star}$. More generally, one can write:
	\begin{align}\label{eq_rec_path}
		f_i^{\star} = \min_{0\leq s\leq i}\left\{q^\star_{[s+1:i]}+f_{s-1}^{\star}\right\}, \qquad f_{u}^{\star} = 0\ \text{for $u\leq 0$}.
	\end{align}
	The values of $\{q_{[k:l]}\}_{k\leq l}$ can be computed in $\mathcal{O}(n^2)$ according to~\cite[Proposition 2]{liu2023graph}. Given the values of $\{q_{[k:l]}\}_{k\leq l}$, the values of $f^\star_{1}, f^\star_{2}, \dots, f^\star_n$ can be obtained via the recursive equation~\eqref{eq_rec_path} in $\mathcal{O}(n^2)$. Consequently, an overall $\mathcal{O}(n^2)$ algorithm for solving~\eqref{eq: MIQP_path} emerges. The corresponding optimal solution can also be recovered with a negligible overhead (see~\cite{liu2023graph} for more details).
	
	A similar DP approach can be extended to trees beyond paths. This extension is particularly viable due to trees inheriting a similar decomposability property: when $z_s=0$ for some node $s$ in the tree, Problem~\eqref{eq: MIQP} decomposes into smaller sub-problems defined over the sub-trees, each rooted at one of the children of $s$, along with a simple quadratic program over the remaining nodes of the tree. Unfortunately, our next example illustrates that this decomposability property is not enough to guarantee the efficiency of the corresponding DP, especially when the tree possesses multiple branches.
	
	\begin{example}[Extended star trees]\label{exp_simple_graph}
		Consider an extended star tree with only one branch located at the root. Let $B$ denote the number of branches in the tree, each composed of $L$ nodes. We define a vector
  $s\in \{0,\ldots, L\}^B$, where each $s_b$ is either $0$ or it corresponds to the $s_b$-th node in branch $b$, with the indices increasing away from the leaf node. If $s_b=0$, then the vector $s$ excludes any node from branch $b$. Figure~\ref{fig:Intro to simple tree} depicts the structure of this graph.
		
		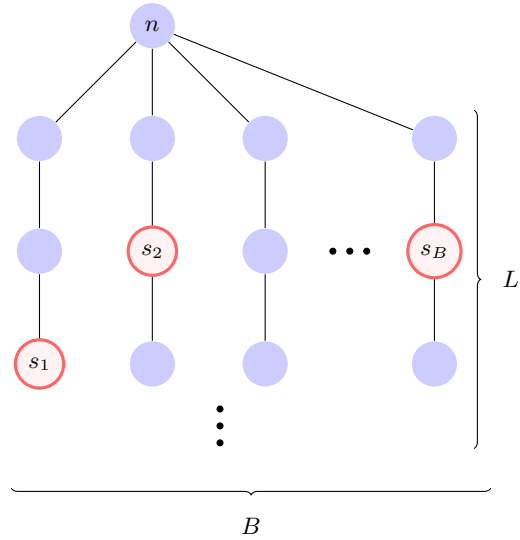
\begin{figure}
			\centering
			\begin{tikzpicture}  
				[scale=.75,auto=center,every node/.style={circle,fill=blue!20,minimum size=6mm},roundnode/.style={circle, draw=red!60, fill=red!5, very thick, minimum size=6mm},blanknode/.style={circle, minimum size=5mm,fill opacity=0,text opacity=1}] 

				\node (a3) at (0,-2)  {$n$}; 
				\node (a4) at (-2,-4)  {};  
				\node (a5) at (-2,-6) {};  
				\node (a6)[roundnode] at (-2,-8)  {\small$s_1$};

				\node (a7) at (0,-4)  {};
				\node (a8)[roundnode] at (0,-6)  {\small $ s_2 $};
				\node (9) at (0,-8)  {};
				
				\filldraw[black] (3.2,-6) circle (1.5pt);
				\filldraw[black] (3.5,-6) circle (1.5pt);
				\filldraw[black] (3.8,-6) circle (1.5pt);
				
				\filldraw[black] (1.2,-8.8) circle (1.5pt);
				\filldraw[black] (1.2,-9.1) circle (1.5pt);
				\filldraw[black] (1.2,-9.4) circle (1.5pt);
				
				\node (a10) at (5,-4)  {};
				\node (a11)[roundnode] at (5,-6)  {\small$ s_B $};
				\node (a12) at (5,-8)  {};	
				
				\node (a13) at (2,-4)  {};
				\node (a14) at (2,-6)  {};
				\node (a15) at (2,-8)  {};	
				
				\draw (a8) -- (9);   
				\draw (a3) -- (a4);  
				\draw (a4) -- (a5);  
				\draw (a5) -- (a6);  
				\draw (a3) -- (a7);  
				\draw (a7) -- (a8);  
				\draw (a13) -- (a14); 
				\draw (a14) -- (a15);  
				\draw (a3) -- (a13);
				
				\draw (a3) -- (a10);  
				\draw (a10) -- (a11);  
				\draw (a11) -- (a12); 
				\draw[decoration={brace,raise=15pt},decorate](5,-3.5) -- (5,-9.5)node[fill=white,text opacity=1,pos=0.5,right=20pt]{$L$};
				\draw[decoration={brace,mirror,raise=15pt},decorate](-2.5,-9.5) -- (6,-9.5)node[fill=white,text opacity=1,pos=0.5,below=20pt]{$B$};
			\end{tikzpicture} \\
			\caption{Simple tree with one branch at its root. A possible choice of $s$ is highlighted in red.}\label{fig:Intro to simple tree}
		\end{figure}
		
		For any $s\in \{0,\ldots, L\}^B$ and $b\in \{1,\ldots, B\}$, let $f_{s_b,b}^\star$ denote the optimal cost of the sub-problem defined over the sub-tree rooted at node $s_b$ within branch $b$. Since this sub-tree is a path, the corresponding $f_{s_b,b}^\star$ can be obtained efficiently via the aforementioned DP algorithm. We set $f_{u,b}^\star=0$ for every $u\leq 0$. Let $s\in \{0,\ldots, L\}^B$ denote the set of nodes with the largest indices in each branch such that $z^\star_s=0$. Accordingly, let $q^\star_{{s}}$ denote the optimal cost of Problem~\eqref{eq: MIQP} with additional constraints that $z_i=0$ for all nodes $i$ within the sub-tree rooted at node $s_b$ for branches $b=1,\dots, B$, and $z_i=1$ otherwise.
		The optimal cost $f^*$ can be written as:
		\begin{align}\label{eq: decomposed problem}
			f^{\star} = \min_{s\in \{0,\dots, L\}^B}\left\{q^\star_{s}+\sum_{b=1}^B f^{\star}_{s_b-1,b}\right\}.
		\end{align}
		The aforementioned equation implies that, even if $f^{\star}_{s_b-1,b}$ and $q^\star_{s}$ can be obtained efficiently,
		one needs to perform up to $(L+1)^B$ comparisons to determine the optimal cost $f^\star$, a process that becomes inefficient with the increasing number of branches. 
	\end{example}
	
	To address the inefficiency inherent in the direct DP approach when applied to general tree structures, we introduce a {parametric characterization} of the optimal cost for Problem~\eqref{eq: MIQP}. Through this characterization, we demonstrate a significant reduction in the search space of the DP approach, without sacrificing the optimality of the obtained solution.
	Toward this goal, in Section~\ref{sec:path}, we revisit Problem~\eqref{eq: MIQP} for path graphs. Here, we present a parametric algorithm for this problem with runtime and memory complexities of $\mathcal{O}(n^2)$, matching the worst-case complexity of the DP approach proposed in~\cite{liu2023graph}, but significantly outperforming it in practice.  
	Building upon our parametric algorithm for path graphs, in the remainder of Section~\ref{sec:alg}, we extend our approach to general tree structures, showing that it can solve these problems in a similar $\mathcal{O}(n^2)$ time and memory. In Section~\ref{sec:implementation}, we discuss an important practical consideration regarding our algorithm.
 Finally, in Section~\ref{sec:experiments}, we assess the performance of our proposed approach across various case studies. Surprisingly, while its worst-case complexity is $\mathcal{O}(n^2)$, we observe that the practical runtime of our proposed algorithm is close to linear in our computations on synthetic and real-world accelerometer datasets.

	\section{Parametric algorithm} \label{sec:alg}
{We first provide a high-level intuition of the proposed algorithm. Recall the definition of $f(\alpha)$, which we repeat for convenience:
  \begin{subequations}\label{eq: MIQP_uRepeated}
		\begin{align}
			f_u(\alpha):=\min_{x\in\bbbr^{n_u},z\in\{0,1\}^{n_u}}\qquad& \dfrac{1}{2}x^\top Q_{[u]}x+{c^\top_{[u]} x}+{\lambda^\top_{[u]}} z\\ 
			\text{s.t.}\qquad &x_i(1-z_i)=0 \qquad \qquad  i=1,2\ldots, n_u\\
			& x_{u}=\alpha.&
		\end{align}
	\end{subequations}
 Note that, as the value function of a mixed-integer problem, the parametric cost $f_u$ is not convex. 
 Nonetheless, a key observation is that, since the support graph is a tree, once the value of $x_{u}$ is fixed, the problem decomposes into independent subproblems, one for each parent of $u$. More specifically, Problem \eqref{eq: MIQP_uRepeated} reduces to
		\begin{align}	f_u(\alpha)=& \frac{1}{2}Q_{u,u}\alpha^2+c_{u}\alpha+\lambda_u\bbbone_{\alpha}+\sum_{v\in \text{par}(u)}\min_{\xi\in \bbbr}\left\{f_{v}(\xi)+Q_{uv}\xi\alpha\right\}\notag\\
        =& \frac{1}{2}Q_{u,u}\alpha^2+c_{u}\alpha+\lambda_u\bbbone_{\alpha}-\sum_{v\in \text{par}(u)}f_v^*(-Q_{uv}\alpha).\label{eq:recursion}
		\end{align}
Since the parametric cost $f_u$ can be characterized merely based on the conjugate of the parametric cost of its parents, we can imagine an algorithm that traverses the graph in topological order, and recursively computes and stores in each node $u$ either the parametric cost $f_u$ or its conjugate $f_u^*$. These functions turn out to be piece-wise quadratic, as elaborated in the following definition.}

	\begin{definition}
		A continuous function $f: \bbbr\to\bbbr$ is termed \textbf{piece-wise quadratic} with $N$ pieces if there exist scalars $-\infty=\tau_0< \tau_1<\dots<\tau_{N}=+\infty$ (also referred to as breakpoints) and quadratic functions (also referred to as pieces) $p_1, \dots, p_{N}$ such that $f(\alpha) = p_k(\alpha)$ for $\tau_{k-1}\leq \alpha\leq \tau_k$, where $k=1,\dots, N$ and $p_k(\alpha)\ne p_{k+1}(\alpha)$ for some $\alpha\in \bbbr$.
	\end{definition}

Upon assuming $p_k(\alpha)=\gamma_{k,1}\alpha^2+\gamma_{k,2}\alpha+\gamma_{k,3}; k=1,\ldots,N$, the condition $p_k(\alpha)\ne p_{k+1}(\alpha)$ for some $\alpha\in \bbbr$ is equivalent to $\gamma_{k,i}\ne\gamma_{k+1,i}$ for some $i\in\{1,2,3\}$. Moreover, to store and represent a piece-wise quadratic function with $N$ pieces, it suffices to store an ordered list of the breakpoints, along with the coefficients of their corresponding quadratic pieces $[(\tau_k, \gamma_{k,1}, \gamma_{k,2}, \gamma_{k,3})]_{k=1}^N$.

{Equation \eqref{eq:recursion} involves sums of value functions, thus the next lemma is critical to our analysis.}
	\begin{lemma}\label{lem_sum_g}
		Consider $L$ piece-wise quadratic functions $\{f_l\}^L_{l=1}$, each with a set of breakpoints $\Gamma_l$. The function $g = \sum_{l=1}^L f_l$ is a piece-wise quadratic function with breakpoints belonging to $\bigcup_{l=1}^L \Gamma_l$.
	\end{lemma}
	\begin{proof}
		Let $-\infty=\tau_0<\tau_1<\dots<\tau_N=+\infty$ be the ordered elements of $\bigcup_{l=1}^L \Gamma_l$. The proof follows by noting that none of $\{f_l\}^L_{l=1}$ contain any breakpoints within the interval $(\tau_{k-1},\tau_k); k=1,\dots, N$. Therefore, the set of breakpoints of $g$ can only belong to $\{\tau_0,\dots, \tau_N\}$. \qed
	\end{proof}

{In order to design an efficient algorithm, recursive equations of the form \eqref{eq:recursion} need to be obtained efficiently. A property that will allow us to do so is the notion of consistency, defined next.}

\begin{definition}\label{def_consistent}
		A piece-wise quadratic function $f$ with $N$ pieces $p_1, \dots, p_{N}$ is called \textbf{consistent} if:
		\begin{enumerate}
			\item $p_1, \dots, p_{N}$ are strongly convex;
			\item $f(\alpha)=\min\limits_{1\le k \le N}\{p_k(\alpha)\}$ for all $\alpha\in \bbbr$.
			\end{enumerate} 
	\end{definition}
Figure~\ref{fig:enter-label} depicts two instances of piece-wise quadratic functions, with only one being consistent.

\begin{figure}
    \centering
    \includegraphics[scale=0.5]{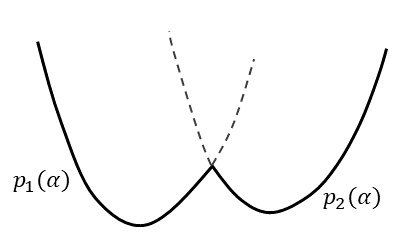}\includegraphics[scale=0.5]{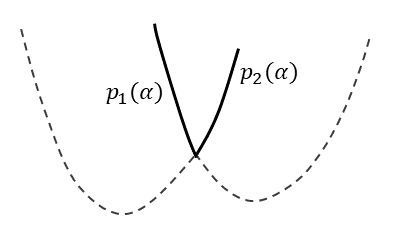}
    \caption{Both functions are piece-wise quadratic. The left figure is consistent, while the right figure is not consistent as it violates the second condition outlined in Definition~\ref{def_consistent}.}
    \label{fig:enter-label}
\end{figure}

{From recursion \eqref{eq:recursion} we see that the algorithm requires the computation of conjugate functions of piece-wise quadratic functions with an indicator variable. Naturally, the overall complexity of the algorithm depends on the number of pieces required to represent the conjugate functions. The next proposition, whose proof we defer to Section~\ref{subsec_g}, shows that the number of pieces can increase by at most 2.}
\begin{proposition}\label{prop_g}
 Consider $f=\tilde f+\lambda\bbbone_{\alpha}$, where $\lambda>0$ and $\tilde f$ is consistent with $N$ pieces. 
There exist an integer $M\leq N+2$ and scalars $-\infty= \tau_0<\tau_1<\dots<\tau_{M}=+\infty$ such that the conjugate function $f^*$ can be written as
		\begin{align}
			& f^*(\beta) = q_{k}(\beta), \qquad \text{for}\quad \tau_{k-1}\leq \beta\leq\tau_k ;\ k=1,\ldots,M,\label{eq_h}
		\end{align}
		where
		\begin{enumerate}
			\item $q_1, \dots, q_M$ are quadratic and convex;
			\item $q_k(\beta)\not=q_{k+1}(\beta)$ for some $\beta$, for $k=1,\dots,M-1$;
			\item $f^*(\beta)=\max\limits_{1\le k \le M}\{q_{k}(\beta)\}$ for all $\beta\in \bbbr$. 
		\end{enumerate}
	\end{proposition}

 At a high level, Proposition~\ref{prop_g} follows from the geometric interpretation of conjugate functions. For simplicity, let us assume that $\lambda=0$. Recall that for any strongly convex and quadratic function $p_k$, its conjugate $p_k^*$ is likewise strongly convex and quadratic. Moreover, $-p_k^*(\beta)$ corresponds to the intercept of a tangent to $p_k$ with slope $\beta$. For any $\beta\in\bbbr$, let $I(\beta)$ denote the minimum index of the piece at which a tangent to $f$ with slope $\beta$ intersects $f$.
The proof of the above proposition relies on two key points: (1) $f^*(\beta) = p^*_{I(\beta)}(\beta)$ for every $\beta\in\bbbr$; and (2) $I(\beta)$ is a non-decreasing function of $\beta$. The first observation implies that $f^*$ is also piece-wise quadratic. The second observation suggests that $I(\beta)$ can have at most $N$ changes, or equivalently, $f^*$ can possess at most $N$ pieces (the additional two pieces in Proposition~\ref{prop_g} arise only if $\lambda>0$). Figure~\ref{fig: Indexing Function} depicts this intuition on a simple consistent function.

\begin{figure}
    \centering
    \includegraphics[scale=0.35]{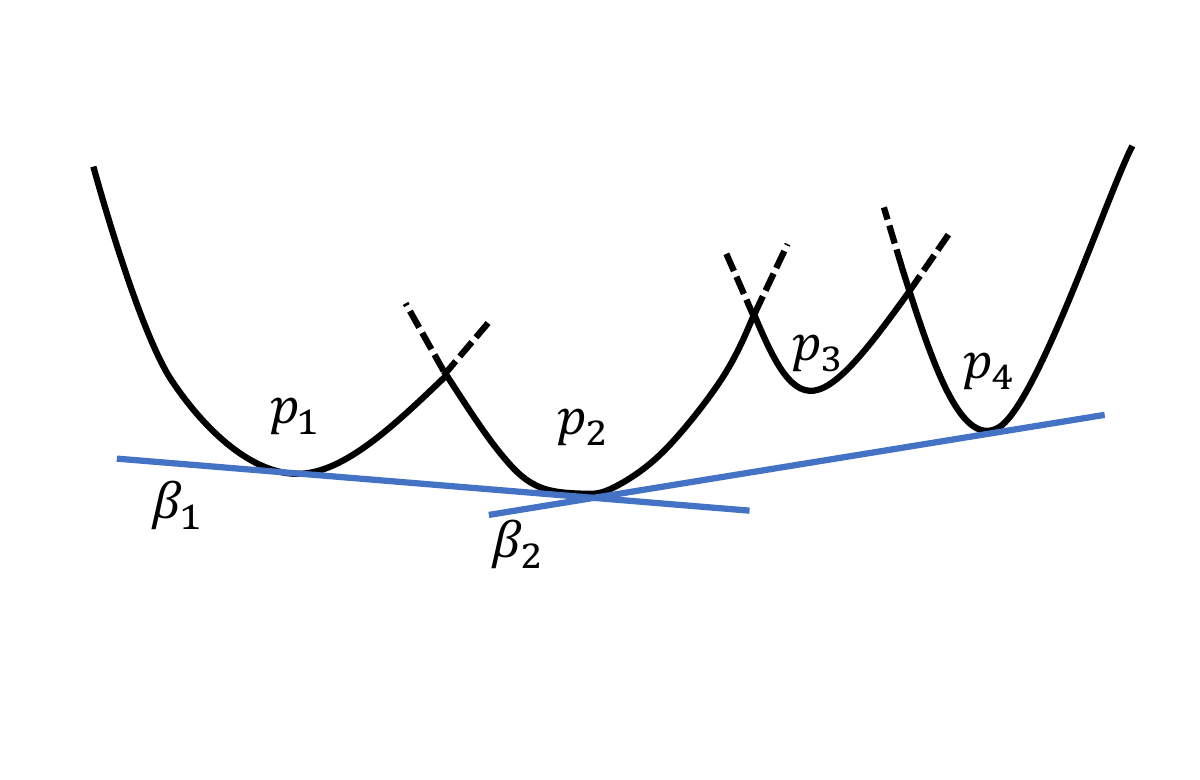}
    \caption{A consistent function $f(\alpha)$ with four strongly convex quadratic pieces. For this function, we have $f^*(\beta) = p_1^*(\beta)$ for $\beta\leq \beta_1$, $f^*(\beta) = p_2^*(\beta)$ for every $\beta_1\leq \beta< \beta_2$, and $f^*(\beta) = p_4^*(\beta)$ for every $\beta_2\leq \beta$. As a result, $f^*(\beta)$ has three pieces.}
    \label{fig: Indexing Function}
\end{figure}

{A few observations are in order based on the above proposition. First, the conjugate function $f^*$ is not guaranteed to be consistent, even if $f$ is consistent (a property that holds when $\lambda=0$). Second, the number of pieces of the conjugate can, in fact, decrease. Intuitively, by computing the conjugate, we implicitly compute the closed convex envelope of the function $f$, that is, we compute and only store the information relevant for optimization instead of the complete function. While, in general, convex envelopes are notoriously hard to compute, the next proposition states that they can be obtained efficiently for consistent functions.}
 
	\begin{proposition}\label{prop:g_efficient}
		Given $f=\tilde f+\lambda\bbbone_{\alpha}$, where $\lambda>0$ and $\tilde f$ is consistent with $N$ pieces,  the conjugate function $f^*$ can be obtained in $\mathcal{O}(N)$ time and memory.
	\end{proposition}
 The proof of Proposition~\ref{prop:g_efficient} is presented in Section~\ref{subsec:breakpoint}. Equipped with these results, we are now ready to present our parametric algorithm for path graphs.
	
	\subsection{Path graphs}\label{sec:path}
	Assume $\supp(Q)$ is a path graph. The following lemma is a direct consequence of Propositions~\ref{prop_g} and~\ref{prop:g_efficient}. It characterizes the parametric cost at every node $u$ based on the parametric cost of its parent node $u-1$.

	\begin{lemma}\label{lem:f_path2}
		Suppose that $\supp(Q)$ is a path graph. Moreover, given any node $u$, suppose that $f_{u-1} = \tilde f_{u-1}+\lambda_{u-1}\bbbone_{\alpha}$, where $\tilde f_{u-1}$ is consistent with $N$ pieces. Then, we can express $f_{u} = \tilde f_u+\lambda_{u}\bbbone_{\alpha}$, where $\tilde f_u$ is consistent with at most $N+2$ pieces.
		Moreover, given $f_{u-1}$, $f_u$ can be found in $\mathcal{O}(N)$ time and memory.
	\end{lemma}
	\begin{proof}
	Since $\tilde f_{u-1}$ is consistent with $N$ pieces, due to Proposition~\ref{prop_g}, there exist an integer $M\leq N+2$ and scalars $-\infty=\tau_{u-1,0}<\tau_{u-1,1}<\dots< \tau_{u-1,M}=+\infty$ such that the conjugate function $f_{u-1}^*$ can be written as:
		$$
		f_{u-1}^*(\beta) = q_{u-1,k}(\beta), \qquad \text{for }\quad \tau_{u-1,k-1}< \beta\le \tau_{u-1,k};\ k=1,\ldots,M,
		$$
		where 
		\begin{enumerate}
			\item $q_{u-1,1}, \dots, q_{u-1,M}$ are quadratic and convex;
			\item $q_{u-1,k}(\beta)\not=q_{u-1,k+1}(\beta)$ for some $\beta$, for $k=1,\dots,M-1$;
			\item $f_{u-1}^*(\beta)=\max\limits_{1\le k \le M}\{q_{u-1,k}(\beta)\}$ for all $\beta\in \bbbr$. 
		\end{enumerate}
		Combined with~\eqref{eq:recursion}, this implies that $f_u = \tilde f_u+\lambda_u\bbbone_{\alpha}$, where
		\begin{multline}\label{eq_q}
			\tilde f_u(\alpha) = \underbrace{\dfrac{1}{2}\alpha^2+c_u\alpha- q_{u-1,k}(-Q_{u,u-1}\alpha)}_{:=p_{u,k}(\alpha)},\\ 
			\text{for }\quad \underbrace{-\frac{ \tau_{u-1,k}}{Q_{u,u-1}}}_{:=\tau_{u,k-1}}< \alpha\le  \underbrace{-\frac{\tau_{u-1,k-1}}{Q_{u,u-1}}}_{:=\tau_{u,k}};\ k=1,\ldots,M,
		\end{multline}
		where we used $Q_{u,u-1}\not=0$ since $\supp(Q)$ is assumed to be connected. Next, we establish that $\tilde f_u$ is indeed consistent. First, the strong convexity of $p_{u,k}$ for $k=1,\dots, M$ directly follows from Lemma~\ref{lem:f_quad}. Second, we observe that $p_{u,k}(\alpha)\not=p_{u,k+1}(\alpha)$ for some $\alpha$ since $q_{u-1,k}(\beta)\not=q_{u-1,k+1}(\beta)$ for some $\beta$. Third, we have 
		\begin{align*}
			\tilde f_u(\alpha) &=  \dfrac{1}{2}\alpha^2+c_u \alpha-f_{u-1}^*(-Q_{u,u-1}\alpha)\\
			&=\dfrac{1}{2}\alpha^2+c_u\alpha-\max\limits_{1\le k \le M}\{ q_{u-1,k}(-Q_{u,u-1}\alpha)\} \\
			&= \min\limits_{1\le k \le M}\left\{\dfrac{1}{2}\alpha^2+c_u \alpha- q_{u-1,k}(-Q_{u,u-1}\alpha)\right\}\\
			&= \min\limits_{1\le k \le M}\{p_{u,k}(\alpha)\}.
		\end{align*}
		Finally, due to Proposition~\ref{prop:g_efficient}, $f^*_{u-1}$ can be obtained in $\mathcal{O}(N)$ time and memory. Combined with~\eqref{eq_q}, this indicates that $f_u$ can also be computed in $\mathcal{O}(N)$ time and memory.\qed
		\end{proof}
	Due to~\eqref{eq:recursion}, the parametric cost $f_1$ at the leaf node 1 is the sum of an indicator function and a consistent function with $N=1$ piece. Therefore, Lemma~\ref{lem:f_path2} implies that $f_{2}$ is the sum of an indicator function and a consistent function with $N\leq 3$ pieces. Moreover, it can be computed in $O(1)$ time. Repeating this process until reaching the root node proves that $f_n$ can be expressed as the sum of an indicator function and a consistent function with $N\leq 2n$ pieces, and it can be computed in $\mathcal{O}(1+3+5+\dots+2n) = \mathcal{O}(n^2)$. Once $f_n$ is determined, the optimal cost $f^\star$ can be derived by minimizing $f_n$ over at most $2n$ strongly convex and quadratic pieces. The details of this procedure are delineated in Algorithm~\ref{alg:path}.
	\begin{algorithm}[h]
		\caption{Parametric algorithm for path graphs}\label{alg:path}
		\textbf{Input:} $ c,\lambda\in \bbbr^n, Q\in \bbbr^{n\times n}$, where $Q$ is positive definite and $\supp(Q)$ is a path graph\\
		\textbf{Output:} The optimal solution $x^\star$ and optimal cost $f^\star$
		\begin{algorithmic}[1]
			\For{$u=1,\dots, n$}
			\State Obtain $f_{u}$ based on $f^*_{u-1}$ via Equation~\eqref{eq:recursion}
			\State Obtain $f^*_{u}$ based on $f_{u}$ via the breakpoint algorithm (Algorithm~\ref{alg: breakpoint})
			\State $u\leftarrow \child(u)$
			\EndFor
			\State Obtain $f^\star=\min\limits_{\alpha} f_n(\alpha)$ and $x^\star_{n}=\argmin\limits_{\alpha} f_n(\alpha)$ 
			\For{$u=n-1,\dots, 1$} 
			\State Set $x^\star_{u}=\argmin\limits_{\alpha}\{f_{u}(\alpha)+Q_{u+1,u}x^\star_{u+1}\cdot \alpha\}$
			\EndFor
			\State\Return $f^\star$ and $x^\star$
		\end{algorithmic}
	\end{algorithm}
	
	\begin{theorem}\label{thm:path}
		Algorithm~\ref{alg:path} solves Problem~\eqref{eq: MIQP_path} in $\mathcal{O}(n^2)$ time and memory.
	\end{theorem}
	\begin{proof}
 Due to~\eqref{eq:recursion}, the parametric cost at the leaf node $1$ can be written as $f_1=\tilde f_1+\lambda_1\bbbone_\alpha$, where $\tilde f_1$ is consistent with $N=1$ piece. Consequently, by inductively applying Lemma~\ref{lem:f_path2} from the leaf to the root node, the correctness of Algorithm~\ref{alg:path} is established.
 
 To show its runtime, we consider the operations within the loops. Since $f^*_{u-1}$ has at most $2n$ pieces, the first operation inside the loop (Line 3) can be executed in $O(n)$ time. Moreover, the second operation inside the first loop (Line 3) can be executed in $O(N) = O(n)$ time due to Proposition~\ref{prop:g_efficient}. Hence, the first loop can be executed in $O(n^2)$ time.
		On the other hand, according to Lemma~\ref{lem:f_path2}, $f_n = \tilde f_n+\lambda_n\bbbone_{\alpha}$, where $\tilde f_n$ is consistent with at most $2n$ pieces. Therefore, Line 6 can be executed in $O(n)$ time by minimizing at most $2n$ strongly convex and quadratic functions. Similarly, each operation inside the second loop can be executed in $O(n)$ time, resulting in $O(n^2)$ time and memory for the second loop. \qed
	\end{proof}

	\subsection{Tree graphs}\label{sec:tree}
	In this section, we extend our parametric algorithm to the general tree structures. Toward this goal, we first revisit Example~\ref{exp_simple_graph} to elucidate the key ideas behind this extension.   
	\setcounter{example}{0}
	\begin{example}[Continued]
		To obtain the optimal cost, akin to the path graphs, it suffices to derive the parametric cost $f_n$. This can be achieved by noting that:
		\begin{align*}
			f_n(\alpha) &= \dfrac{1}{2}\alpha^2+c_n\alpha+\lambda_n\bbbone_{\alpha}+\sum_{v\in\parent(n)}{\min_{\xi}\left\{Q_{n,v}\alpha\cdot \xi+f_v(\xi)\right\}}\\
			&=  \dfrac{1}{2}\alpha^2+c_n\alpha+\lambda_n\bbbone_{\alpha}-\sum_{v\in\parent(n)}{\max_{\xi}\left\{-Q_{n,v}\alpha\cdot \xi-f_v(\xi)\right\}}\\
			&=\dfrac{1}{2}\alpha^2+c_n\alpha+\lambda_n\bbbone_{\alpha}-\sum_{v\in\parent(n)}f^*_v\left(-Q_{n,v}\alpha\right).
		\end{align*}
		For every $v\in\parent(n)$, $\supp_v(Q)$ is a path. Therefore, according to our discussion in the previous section, each $f^*_v$ is consistent with at most $2L+2$ pieces, and can be obtained in $O(L^2)$ time via Algorithm~\ref{alg:path}. Therefore, $\sum_{v\in\parent(n)}f_v^*\left(-Q_{n,v}\alpha\right)$ can be obtained in $\mathcal{O}(BL^2)$. On the other hand, invoking Lemma~\ref{lem_sum_g} implies that $\sum_{v\in\parent(n)}f_v^*\left(-Q_{n,v}\alpha\right)$ is a piece-wise quadratic function with at most $B(2L+2)$ pieces. 
		Therefore, the optimal cost $f^\star$ can be obtained by minimizing different pieces of $f_n$ in $\mathcal{O}(BL)$. This brings the complexity of the parametric algorithm to $\mathcal{O}(BL^2)$. This is a significant improvement upon the direct DP approach, which runs in $\mathcal{O}\left((L+1)^B\right)$. 
	\end{example}
	Motivated by the above example, we next present the analog of Lemma~\ref{lem:f_path2} for tree graphs.

	\begin{lemma}\label{lem:f_tree2}
		Suppose that $\supp(Q)$ is a tree graph. Moreover, given any node $u$, suppose that $f_{v} = \tilde f_{v}+\lambda_{v}\bbbone_{\alpha}$ for every $v\in \parent(u)$, where $\tilde f_{v}$ is consistent with $N_v$ pieces. Then, we can express $f_{u} = \tilde f_u+\lambda_{u}\bbbone_{\alpha}$, where $\tilde f_u$ is consistent with at most $\sum_{v\in\parent(u)}(N_v+2)$ pieces.
		Moreover, given $\{f_{v}\}_{v\in\parent(u)}$, $f_u$ can be found in $\mathcal{O}\left(\sum_{v\in\parent(u)}N_v\right)$ time and memory.
	\end{lemma}
	\begin{proof}
      Since for every $v\in\parent(u)$, $\tilde f_{v}$ is consistent with $N_v$ pieces, Proposition~\ref{prop_g} implies the existence of an integer $M_v\leq N_v+2$ and scalars $-\infty=\tau_{v,0}<\tau_{v,1}<\dots< \tau_{v,M_v}=+\infty$ such that $f_{v}^*$ can be written as:
		$$
		f_{v}^*(\beta) = q_{v,k}(\beta), \qquad \text{for }\quad \tau_{v,k-1}< \beta\le \tau_{v,k};\ k=1,\ldots,M_v,
		$$
		where 
		\begin{enumerate}
			\item $q_{v,1}, \dots, q_{v,M}$ are quadratic and convex;
			\item $q_{v,k}(\beta)\not=q_{v,k+1}(\beta)$ for some $\beta$, for $k=1,\dots,M_v-1$;
			\item $f_{v}^*(\beta)=\max\limits_{1\le k \le M_v}\{q_{v,k}(\beta)\}$ for all $\beta\in \bbbr$. 
		\end{enumerate}
		Let $\Gamma_v$ be the ordered list of the breakpoints of $f_{v}^*(-Q_{u,v}\alpha)$ defined as $\Gamma_v = \{-\tau_{v,k}/Q_{u,v}\}_{k=1}^{M_v}$. 
		Consider $g_u(\alpha) = \sum_{v\in\parent(u)}f_v^*(-Q_{u,v}\alpha)$. According to Lemma~\ref{lem_sum_g}, $g_u$ is piece-wise quadratic with a set of breakpoints $\bigcup_{v\in\parent(u)}\Gamma_v$ that has a cardinality of $N_u\leq 1+\sum_{v\in\parent(u)}(N_v+2)$. Given the ordered lists $\{\Gamma_v\}_{v\in \parent(v)}$, $\bigcup_{v\in\parent(u)}\Gamma_v$ can be ordered and stored in $\mathcal{O}\left(\sum_{v\in\parent(u)}N_v\right)$ time and memory. Let $-\infty=\tau_{u,0}<\tau_{u,1}<\dots<\tau_{u, N_u}=+\infty$ be the ordered elements of $\bigcup_{v\in\parent(u)}\Gamma_v$. One can write 
		\begin{align}\label{eq_g_tree}
			g_u(\alpha) = \underbrace{\sum_{v\in\parent(u)}q_{v,i_v(k)}(-Q_{u,v}\alpha)}_{:=\tilde q_{u,k}(\alpha)},\quad
			\text{for }\quad \tau_{u,k-1}< \alpha\le  \tau_{u,k};\ k=1,\ldots,N_u,
		\end{align}
		where $i_v(k)$ is the index for which $[\tau_{u,k-1}, \tau_{u,k}]\subseteq \left[-\frac{\tau_{v,i_v(k)}}{Q_{u,v}}, -\frac{\tau_{v,i_v(k)-1}}{Q_{u,v}}\right]$ if $Q_{u,v}>0$ and $[\tau_{u,k-1}, \tau_{u,k}]\subseteq \left[-\frac{\tau_{v,i_v(k)-1}}{Q_{u,v}}, -\frac{\tau_{v,i_v(k)}}{Q_{u,v}}\right]$ if $Q_{u,v}<0$. 
		The above equation combined with~\eqref{eq:recursion} implies that $f_u = \tilde f_u+\lambda_u\bbbone_{\alpha}$, where
		\begin{align}\label{eq_q_tree}
			\tilde f_u(\alpha) = \underbrace{\dfrac{1}{2}\alpha^2+c_u\alpha- \tilde q_{u,k}(\alpha)}_{:=p_{u,k}(\alpha)},\quad
			\text{for }\quad \tau_{u,k-1}< \alpha\le  \tau_{u,k};\ k=1,\ldots,N_u.
		\end{align}
		Next, we prove that $\tilde f_u$ is consistent. First,  if $p_{u,k}$ and $p_{u,k+1}$ are identical for some $1\leq k\leq N_u$, one can remove the $(k+1)$-th piece and set $\tau_{u,k} \leftarrow \tau_{u,k+1}$ and $N_u\leftarrow N_u-1$. This process can be repeated until $p_{u,k}$ and $p_{u,k+1}$ are not identical for all $1\leq k\leq N_u-1$. Second, the strong convexity of $p_{u,k}$ for $k=1,\dots, N_u$ directly follows from Lemma~\ref{lem:f_quad}. Third, note that
		\begin{align*}
			g_u(\alpha) &= \sum_{v\in\parent(u)}f_v^*(-Q_{u,v}\alpha)\\
			&=\sum_{v\in\parent(u)}\max_{1\leq k\leq N_v}\left\{q_{v,k}(-Q_{u,v}\alpha)\right\}\\
			&\geq \max_{1\leq k\leq N_u}\left\{\sum_{v\in\parent(u)}q_{v,i_v(k)}(-Q_{u,v}\alpha)\right\}\\
			&=\max_{1\leq k\leq N_u}\left\{\tilde q_{u,k}(\alpha)\right\}\\
			&\geq g_u(\alpha),
		\end{align*}
		\begin{sloppypar}
			\noindent where the last inequality follows from~\eqref{eq_g_tree}.
			Therefore, we have $g_u(\alpha) = \max_{1\leq k\leq N_u}\left\{\tilde q_{u,k}(\alpha)\right\}$. This leads to
		\end{sloppypar}

		\begin{align*}
			\tilde f_u(\alpha) &= \dfrac{1}{2}\alpha^2+c_u \alpha - g_u(\alpha)\\
			&=\dfrac{1}{2}\alpha^2+c_u \alpha-\max_{1\leq k\leq N_u}\left\{\tilde q_{u,k}(\alpha)\right\} \\
			&= \min\limits_{1\le k \le N_u}\left\{\dfrac{1}{2}\alpha^2+c_u \alpha- \tilde q_{u,k}(\alpha)\right\}\\
			&= \min\limits_{1\le k \le M}\{p_{u,k}(\alpha)\}.
		\end{align*}
		This completes the proof of the consistency of $\tilde f_u$. Finally, due to Proposition~\ref{prop:g_efficient}, each $f^*_{v}$ can be obtained in $\mathcal{O}\left(N_v\right)$ time and memory. Therefore, $g_u(\alpha) = \sum_{v\in\parent(u)}f_v^*(-Q_{u,v}\alpha)$ can be obtained in $\mathcal{O}\left(\sum_{v\in\parent(u)}N_v\right)$ time and memory. Combined with $f_u(\alpha) = (1/2)\alpha^2+c_u \alpha + \lambda_u\bbbone_{\alpha}- g_u(\alpha)$, this indicates that $f_u$ can also be computed in $\mathcal{O}\left(\sum_{v\in\parent(u)}N_v\right)$ time and memory.\qed
		\end{proof}
	
With Lemma~\ref{lem:f_tree2} in place, we are prepared to present an overview of our parametric algorithm for general tree graphs. The algorithm starts with node $1$. Since node 1 represents a leaf node, its parametric cost $f_1$ can be readily determined based on the recursion \eqref{eq:recursion}. Moreover, its conjugate $f^*_1$ can be obtained in $\mathcal{O}(1)$ due to Proposition~\ref{prop:g_efficient}. Assuming that the parametric costs $f_v$ and their conjugates $f^\star_v$ are available for every node $v<u$, the parametric cost $f_u$, can be obtained based on Lemma~\ref{lem:f_tree2}. Notably, due to the topological ordering of nodes, all $v\in \parent(u)$ satisfy $v<u$, ensuring that their conjugate parametric costs $f_v^*$ needed to characterize $f_u$ are known. By repeating this process iteratively, the algorithm efficiently computes the parametric costs following the increasing topological ordering.

	Algorithm~\ref{alg: general trees} formalizes the aforementioned intuition and presents the proposed parametric algorithm for trees with greater detail.

 	\begin{algorithm}
		\caption{Parametric algorithm over general trees}\label{alg: general trees}
		\textbf{Input:} $ c,\lambda\in \bbbr^n, Q\in \bbbr^{n\times n}$, where $Q$ is positive definite and $\supp(Q)$ has a tree structure\\
		\textbf{Output:} The optimal solution $x^\star$ and optimal cost $f^\star$
		\begin{algorithmic}[1]
			\State Label the nodes $\supp(Q)$ according to their topological ordering
   \For{$u=1,\dots,n$}
\State Obtain $f_{u}$ based on $\{f^*_{v}\}_{v\in \parent(u)}$ via Equation~\eqref{eq:recursion}
\State Obtain $f^*_{u}$ from $f_{u}$ via Algorithm~\ref{alg: breakpoint}
\EndFor
			\State Obtain $f^\star=\min\limits_{\alpha} f_n(\alpha)$ and $x^\star_{n}=\argmin\limits_{\alpha} f_n(\alpha)$ 
			\State $J \leftarrow\parent(n)$
			\While{$J\ne \{\}$}
			\State Choose $u\in J$
			\State Set $x^\star_u=\argmin\limits_{\alpha}\{f_u(\alpha)+Q_{\child(u),u}x^\star_{\child(u)} \alpha\}$ 
			\State $J\leftarrow J\backslash\{u\}$
			\State $J\leftarrow J\cup\parent(u)$
			\EndWhile
			\State \Return $f^\star$ and $x^\star$
		\end{algorithmic}
	\end{algorithm}
 
	\begin{theorem}
		Under the assumption that $\supp(Q)$ is a tree, Algorithm~\ref{alg: general trees} solves Problem~\eqref{eq: MIQP} in $\mathcal{O}(n^2)$ time and memory.
	\end{theorem}
	\begin{proof}
 The proof is analogous to that of Theorem~\ref{thm:path}, and proceeds inductively using Equation~\eqref{eq:recursion} and Lemma~\ref{lem:f_tree2}. For brevity, we omit the specific details.\qed

	\end{proof}

	\subsection{Properties of consistent functions}\label{subsec_g}
 In this section, we present the proof of Proposition~\ref{prop_g}.
	To this goal, we first introduce the fundamental properties of consistent functions and their conjugates.

For a piece-wise quadratic function $g$ with $N$ strongly convex pieces $p_1,\dots, p_N$, we define its \textit{indexing function} $I_g: \bbbr\to \{1, \dots, N\}$ as:
	\begin{align}\label{eq:Indexing function}
		I_g(\beta) = \min\left\{k: \tau_{k-1}\leq \alpha^\star\leq \tau_{k}, \alpha^\star\in\argmax_\alpha\left\{\beta\alpha-g(\alpha)\right\}\right\},
	\end{align}
 where $\{\tau_{k}\}_{k=0}^N$ are the breakpoints of $g$. Intuitively, the indexing function $I_g$ returns the piece with the minimum index where a line with slope $\beta$ is tangent to $g$. As an example, the indexing function for $f$ depicted in Figure~\ref{fig: Indexing Function} can be characterized as $I_f(\beta)=1$ for all $\beta\leq \beta_1$, $I_f(\beta)=2$ for all $\beta_1<\beta\leq \beta_2$, and $I_f(\beta)=4$ for all $\beta_2<\beta$. 
 
 Due to the definition of the indexing function, there exists a solution $\alpha^\star\in\argmax_\alpha\{\beta\alpha-g(\alpha)\}$ such that $\tau_{I_g(\beta)-1}\leq \alpha^\star\leq \tau_{I_g(\beta)}$. Therefore, we have
\begin{align}\label{eq_hp}
    g^*(\beta) = \beta\alpha^\star-g(\alpha^\star) = \beta\alpha^\star-p_{I_g(\beta)}(\alpha^\star) = \max_\alpha\{\beta\alpha-p_{I_g(\beta)}(\alpha)\} = p^*_{I_g(\beta)}(\beta).
\end{align}
Let the image of $I_g$ be denoted as $\range(I_g) = \{k: k=I_g(\beta) \text{ for some $\beta\in\bbbr$}\}$.  
For every  $k\in \range(I_g)$, its inverse image is defined as $I^{-1}_g(j) = \{\beta: I_g(\beta)=k\}$. Revisiting Figure~\ref{fig: Indexing Function}, the indexing function of $f$ satisfies $\range(I_f) = \{1,2,4\}$ with inverse images $I^{-1}_f(1) = (-\infty,\beta_1]$, $I^{-1}_f(2) = (\beta_1,\beta_2]$, and $I^{-1}_f(4) = (\beta_2,+\infty)$.

Recall the intuition behind Proposition~\ref{prop_g}: In order to control the number of pieces of $g^*$, it suffices to control the number of changes in the indexing function $I_g$. This can be achieved by showing that $I_g$ is non-decreasing. Our next lemma establishes this important property for consistent functions.

\begin{lemma}\label{lem_consistent_non_decreasing}
    Any consistent function has a non-decreasing indexing function.
\end{lemma}
To prove the above lemma, we first present the following intermediate result.

\begin{lemma}\label{lemma: addition prop of f for lemma1}
 Suppose that $g$ is consistent with pieces $p_1, \dots, p_N$ and breakpoints $-\infty=\tau_0<\tau_1<\dots<\tau_N=+\infty$.
		For any $\beta\in\bbbr$ and $k\in\{1,\dots, N\}$, define the linear function $\ell_{k;\beta}(\alpha) = \beta\alpha-p_k^*(\beta)$. Moreover, define $\alpha^\star(\beta)\in\argmax_{\alpha}\{\beta\alpha-g(\alpha)\}$. Let $k^*$ be such that $\tau_{k^*-1}\leq \alpha^\star(\beta)\leq \tau_{k^*}$. The following statements hold:
		\begin{enumerate}
			\item We have $\alpha^\star(\beta)\not\in\{\tau_{k^*-1},\tau_{k^*}\}$.
			\item We have $\ell_{k^*;\beta}(\alpha^\star(\beta)) = g(\alpha^\star(\beta))$, and $\ell_{k^*;\beta}(\alpha)\leq g(\alpha)$ for every $\alpha\in \bbbr$. 
		\end{enumerate}
	\end{lemma}
	\begin{proof}
		To prove the first statement, suppose, by contradiction, that $\alpha^\star(\beta)=\tau_k$ for some $k\in\{k^*-1,k^*\}$. Note that $g(\alpha)-\beta\alpha = p_k(\alpha)-\beta\alpha$ for every $\tau_{k-1}\leq \alpha\leq \tau_k$ and $g(\alpha)-\beta\alpha = p_{k+1}(\alpha)-\beta\alpha$ for every $\tau_{k}\leq \alpha\leq \tau_{k+1}$. Since $\alpha^\star(\beta)=\tau_k$, we must have $p_k'(\alpha^\star(\beta))\leq 0$ and $p_{k+1}'(\alpha^\star(\beta))\geq 0$. Since $\beta\alpha-g(\alpha)$ is a continuous function of $\alpha$, we must have $p_k(\alpha^\star(\beta))=p_{k+1}(\alpha^\star(\beta))$. We consider three cases:
		\begin{enumerate}
			\item {Suppose $p_k'(\alpha^\star(\beta))=p_{k+1}'(\alpha^\star(\beta))=0$.}  Since $p_k$ and $p_{k+1}$ are not identical, we must have $p_k''(\alpha^\star(\beta))\not=p_{k+1}''(\alpha^\star(\beta))$. If $p_k''(\alpha^\star(\beta))< p_{k+1}''(\alpha^\star(\beta))$, then $p_k(\alpha^\star(\beta)+\epsilon)< p_{k+1}(\alpha^\star(\beta)+\epsilon)$ for every $\epsilon>0$, which is a contradiction. Similarly, if $p_k''(\alpha^\star(\beta))> p_{k+1}''(\alpha^\star(\beta))$, then $p_k(\alpha^\star(\beta)-\epsilon)> p_{k+1}(\alpha^\star(\beta)-\epsilon)$ for every $\epsilon>0$, which is again a contradiction.
			\item {Suppose $p_k'(\alpha^\star(\beta))<0$.} Therefore, there exists $\bar\epsilon>0$ such that, for every $\epsilon\in (0,\bar\epsilon]$, we have
			\begin{align*}
				p_k(\alpha^\star(\beta)+\epsilon)<p_k(\alpha^\star(\beta))=p_{k+1}(\alpha^\star(\beta))\leq p_{k+1}(\alpha^\star(\beta)+\epsilon),
			\end{align*}
			which is a contradiction.
			\item {Suppose $p_{k+1}'(\alpha^\star(\beta))>0$.} Therefore, there exists $\bar\epsilon>0$ such that, for every $\epsilon\in (0,\bar\epsilon]$, we have
			\begin{align*}
				p_{k+1}(\alpha^\star(\beta)-\epsilon)<p_{k+1}(\alpha^\star(\beta))=p_{k}(\alpha^\star(\beta))\leq p_{k}(\alpha^\star(\beta)-\epsilon),
			\end{align*}
			which is a contradiction.
		\end{enumerate}
		To prove the second statement, recall that $\beta\alpha-g(\alpha) = \beta\alpha-p_{k^*}(\alpha)$ for every $\tau_{k^*-1}\leq \alpha \leq \tau_{k^*}$. Therefore, since $\alpha^\star(\beta)\in \argmax_{\alpha}\{\beta\alpha-g(\alpha)\}$ and $\alpha^\star(\beta)\in (\tau_{k^*-1},\tau_{k^*})$, we must have $\alpha^\star(\beta)\in \argmin_{\alpha\in (\tau_{k^*-1},\tau_{k^*})}\{p_{k^*}(\alpha)-\beta\alpha\}$. Since $p_{k^*}(\alpha)-\beta\alpha$ is a strongly convex function of $\alpha$, this implies that $\alpha^\star(\beta)=\argmin_{\alpha}\{p_{k^*}(\alpha)-\beta\alpha\}=  \argmax_{\alpha}\{\beta\alpha-p_{k^*}(\alpha)\}$. Therefore,
		\begin{align*}
			\beta\alpha^\star(\beta)-p_{k^*}(\alpha^\star(\beta)) &= p^*_{k^*}(\beta)\\
			\iff \beta\alpha^\star(\beta)-p^*_{k^*}(\beta) &= p_{k^*}(\alpha^\star(\beta))\\
			\iff \ell_{k^*;\beta}(\alpha^\star(\beta)) &=  g(\alpha^\star(\beta)).
		\end{align*}
		Finally, since $\max_{\alpha}\{\beta\alpha-g(\alpha)\} = \max_{\alpha}\{\beta\alpha-p_{k^*}(\alpha)\}$, one can write
		\begin{align*}
			\ell_{k^*;\beta}(\alpha) &= \beta\alpha-p_{k^*}^*(\beta)\\
			&=\beta\alpha-\max_{\xi} \{\beta \xi-p_{k^*}(\xi)\}\\
			&=\beta\alpha+\min_{\xi} \{-\beta \xi+p_{k^*}(\xi)\}\\
			&=\beta\alpha+\min_{\xi} \{-\beta \xi+g(\xi)\}\\
			&=\min_{\xi} \{\beta\alpha-\beta \xi+g(\xi)\}\\
			&\le \beta\alpha-\beta\alpha+g(\alpha)\\
			&\le g(\alpha).
		\end{align*}
		This completes the proof.\qed
	\end{proof}

 \paragraph{Proof of Lemma~\ref{lem_consistent_non_decreasing}.}
Suppose that $g$ is consistent with pieces $p_1, \dots, p_N$ and breakpoints $-\infty=\tau_0<\tau_1<\dots<\tau_N=+\infty$.
 To show $I_g(\beta)$ is non-decreasing, it suffices to show that if $k<I_g(\beta)$ for some $\beta\in \bbbr$, then $k\not=I_g(\beta')$, for any $\beta'>\beta$. By contradiction, suppose there exist $\beta<\beta'$ such that $k<I_g(\beta)$ and $k=I_g(\beta')$. Let $l=I_g(\beta)$. Due to the definition of the indexing function, there exist $\alpha^\star_l, \alpha^\star_k\in \bbbr$ such that
		\begin{align*}
			&\alpha^\star_l\in\argmax_\alpha\{\beta\alpha-g(\alpha)\},\ \text{and}\ \tau_{l-1}\leq \alpha^\star_l\leq \tau_l,\\
			&\alpha^\star_k\in\argmax_\alpha\{\beta'\alpha-g(\alpha)\},\ \text{and}\ \tau_{k-1}\leq \alpha^\star_k\leq \tau_k.
		\end{align*} 
		Due to the first statement of Lemma~\ref{lemma: addition prop of f for lemma1}, we must have $\tau_{l-1}< \alpha^\star_l<\tau_l$ and $\tau_{k-1}< \alpha^\star_k<\tau_k$. This implies that
		\begin{align}\label{eq_oneside}
			\alpha^\star_k<\tau_k\leq \tau_{l-1}< \alpha^\star_l\implies 	\alpha^\star_k<\alpha^\star_l.
		\end{align}
		On the other hand, the second statement of Lemma~\ref{lemma: addition prop of f for lemma1} implies that
		\begin{align*}
			&\ell_{l;\beta}(\alpha^\star_l) = g(\alpha^\star_l)\ \text{and}\ \ell_{l;\beta}(\alpha) \leq g(\alpha); \forall \alpha\\
			&\ell_{k;\beta'}(\alpha^\star_k) = g(\alpha^\star_k)\ \text{and}\ \ell_{k;\beta'}(\alpha) \leq g(\alpha); \forall \alpha.
		\end{align*}
		Combining the above two inequalities, we have
		\begin{align*}
			&\ell_{l;\beta}(\alpha^\star_k) \leq  g(\alpha^\star_k) = \ell_{k;\beta'}(\alpha^\star_k)\implies \beta\alpha^\star_k-p^*_l(\beta)\leq  \beta'\alpha^\star_k-p^*_k(\beta')\\
			&\ell_{k;\beta'}(\alpha^\star_l) \leq  g(\alpha^\star_l) = \ell_{l;\beta}(\alpha^\star_l)\implies \beta'\alpha^\star_l-p^*_k(\beta')\leq  \beta\alpha^\star_l-p^*_l(\beta).
		\end{align*}
		The above two inequalities yield
		\begin{align*}
			&\beta\alpha^\star_k-p^*_l(\beta)+\beta'\alpha^\star_l-p^*_k(\beta')\leq \beta'\alpha^\star_k-p^*_k(\beta')+\beta\alpha^\star_l-p^*_l(\beta)\\
			\iff& (\beta'-\beta)\alpha^\star_l\leq (\beta'-\beta)\alpha^\star_k\\
			\iff& \alpha^\star_l\leq \alpha^\star_k,
		\end{align*}
		which contradicts~\eqref{eq_oneside}. This completes the proof. \qed

Our next lemma provides a key property of the conjugate of a piece-wise quadratic function with a non-decreasing indexing function.
\begin{lemma}\label{lem:h}
		Suppose that $g$ is a piece-wise quadratic function with $N$ strongly convex pieces $p_1,\dots, p_N$ and a non-decreasing indexing function $I_g$. There exist an integer $N'\leq N$, scalars $-\infty= \tau_0<\tau_1<\dots<\tau_{N'}=+\infty$, and a strictly increasing function $\pi: \{1,\dots, N'\}\to \{1,\dots, N\}$ such that
		\begin{align}
			g^*(\beta) = p_{\pi(k)}^*(\beta), \qquad \text{for}\quad \tau_{k-1}\leq \beta\leq\tau_k ;\ k=1,\ldots,N',\label{eq_h2}
		\end{align}
		where
		\begin{enumerate}
			\item $p^*_1, \dots, p^*_N$ are quadratic and strongly convex;
			\item $p^*_{\pi(k)}(\beta)\not=p^*_{\pi(k+1)}(\beta)$ for some $\beta$, for $k=1,\dots,N'-1$.
		\end{enumerate}
	\end{lemma}
 \begin{proof}
     Let $j_1<j_2<\dots<j_{N'}$ be the ordered elements of $\range(I_g)$. We have $N'\leq N$ since $\range(I_g)\subseteq \{1,\dots,N\}$. Moreover, we have $\bigcup_{k=1}^{N'} I^{-1}_g(j_k) = \bbbr$. Since $I_g$ is assumed to be non-decreasing, $I^{-1}_g(j_k)$ is a convex set for every $k=1,\dots,N'$. Therefore, there exist $-\infty=\tau_0< \tau_1\leq \dots\leq \tau_{N'-1}<\tau_{N'}=+\infty$ such that, for every $k=1,\dots,N'$, $I^{-1}_g(j_k)$ can be characterized as: 
	\begin{align}\label{eq_I_inv}
		I^{-1}_g(j_k) = [\tau_{k-1}, \tau_k]\ ,\ (\tau_{k-1}, \tau_k)\ ,\ [\tau_{k-1}, \tau_k), \ \text{or}\ (\tau_{k-1}, \tau_k].
	\end{align}
	Upon defining $\pi(k) = j_k$ for every $1\leq k\leq N'$, we have
	\begin{align*}
		&g^*(\beta) = p^*_{j_k}(\beta)\qquad\ \ \text{if}\qquad \beta\in I^{-1}_g(j_k)\\
		\iff & g^*(\beta) = p^*_{\pi(k)}(\beta)\qquad \text{if}\qquad \beta\in I^{-1}_g(j_k)\\
		\iff & g^*(\beta) = p^*_{\pi(k)}(\beta)\qquad \text{if}\qquad \tau_{k-1}\leq \beta\leq \tau_k; k=1,\dots, N',
	\end{align*}
 where the first equality is a direct consequence of~\eqref{eq_hp}, the second equality is due to the definition of the function $\pi$, and the third equality is due to~\eqref{eq_I_inv} and the fact that $g^*$ is continuous. This completes the proof of~\eqref{eq_h2}. Next, we proceed to prove the properties delineated in Lemma~\ref{lem:h}. To prove the first property, recall that $p_k$ is strongly convex and quadratic for every $k=1,\dots,N$. Therefore, $p^*_k(\beta) = \max_\alpha\{\beta\alpha-p_k(\alpha)\}\}$ is also strongly convex and quadratic. Moreover, the second property follows since, if $p_{\pi(k)}^*$ and $p_{\pi(k+1)}^*$ are identical for some $1\leq k\leq N'$, one can remove the $(k+1)$-th piece and set $\tau_{k} \leftarrow \tau_{k+1}$ and $N'\leftarrow N'-1$. This process can be repeated until $p_{\pi(k)}^*(\beta)$ and $p_{\pi(k+1)}^*(\beta)$ are not identical for all $k=1,\dots, N'-1$.
  \qed
 \end{proof}

We are now ready to present the proof of Proposition~\ref{prop_g}.
	
\paragraph{Proof of Proposition~\ref{prop_g}.} One can write
	\begin{equation}\label{eq_fh}
		\begin{aligned}
			f^*(\beta) &= \max_\alpha\left\{\beta\alpha-f(\alpha)\right\} \\
			&= \max\left\{-f(0),\max_{\alpha\not=0}\left\{\beta\alpha-\tilde f(\alpha)-\lambda\right\} \right\}\\
			&= \max\left\{-f(0), \underbrace{\max_\alpha\left\{\beta\alpha-\tilde f(\alpha)\right\}}_{:=\tilde f^*(\beta)}-\lambda\right\}.
		\end{aligned}
	\end{equation}
Next, note that
\begin{align}\label{eq_h_max}
			\tilde f^*(\beta)&=\max\limits_{\alpha}\left\{\beta\alpha-\tilde{f}(\alpha)\right\}\nonumber\\
			&=\max\limits_{\alpha}\left\{\beta\alpha-\min_{1\le k\le N}\{p_k(\alpha)\}\right\}\nonumber\\
			&=\max\limits_{\alpha}\left\{\max_{1\le k\le N}\{\beta\alpha-p_k(\alpha)\}\right\}\nonumber\\
			&=\max_{1\le k\le N}\left\{\max\limits_{\alpha}\{\beta\alpha-p_k(\alpha)\}\right\}\nonumber\\
			&=\max_{1\le k\le N}\{p_k^*(\beta)\}.
		\end{align}
  Since each $p_k^*$ is strongly convex, $\tilde f^*$ is also strongly convex. Therefore, the equation $\tilde f^*(\beta)-\lambda = -f(0)$ can have at most two solutions. Moreover, $\tilde f^*(0)=-f(0)$ which implies $\tilde f^*(0)-\lambda<-f(0)$. Hence, $\tilde f^*(\beta)-\lambda = -f(0)$ has exactly two solutions.  Let $\beta_1 <\beta_2$ be these solutions. Based on~\eqref{eq_fh}, $f^*$ can be characterized as
	\begin{align}\label{eq_fstar}
		f^*(\beta) = \begin{cases}
			-f(0) & \beta_1\leq \beta\leq \beta_2\\
			\tilde f^*(\beta)-\lambda & \text{otherwise}.
		\end{cases}
	\end{align}
 Since $\tilde f$ is consistent, it must have a non-decreasing indexing function due to Lemma~\ref{lem_consistent_non_decreasing}. Combined with Lemma~\ref{lem:h}, this implies that $\tilde f$ has at most $N'\leq N$ pieces. 
	Therefore, $f^*$ emerges as a piece-wise quadratic function with at most $M$ pieces, where $M\leq N'+2\leq N+2$. Let these pieces be denoted as $\{q_k\}_{k=1}^M$. For every $1\leq k\leq M$, we either have $q_k(\beta) = p_{k'}^*(\beta)-\lambda$ for some $1\leq k'\leq N$, or $q_k(\beta) = -f(0)$. Therefore, $q_1,\dots, q_M$ are quadratic and convex. Moreover, it is easy to verify that $q_k$ and $q_{k+1}$ are not identical for all $k=1,\dots, M-1$. Finally, note that
\begin{align*}
		\tilde f^*(\beta) = \max_{1\le k\le N}\{ p_k^*(\beta)\}\geq \max_{1\le k\le N'}\{ p_{\pi(k)}^*(\beta)\}\geq p^*_{I_{\tilde f}(\beta)}(\beta) = \tilde f^*(\beta),
	\end{align*}
 where the first equality follows from~\eqref{eq_h_max}. The above inequality implies that $\tilde f^*(\beta)=\max_{1\le k\le N'}\{p_{\pi(k)}^*(\beta)\}$. Therefore, according to~\eqref{eq_fh}, we have 
 $$f^*(\beta) = \max\{ -f(0), \tilde f^*(\beta) - \lambda\} = \max_{1\leq k\leq M}\{q_k(\beta)\}.$$ This completes the proof.\qed

	\subsection{Breakpoint algorithm}\label{subsec:breakpoint}
	Our next goal is to characterize $f^*$ efficiently. Indeed,  the function $f^*$ can be expressed as 
	\begin{align*}
	f^*(\beta)\! =\! \max\{-f(0), \tilde f^*(\beta)-\lambda\} &=\! \max\left\{\underbrace{-f(0)}_{\tilde p_0(\beta)},\!\max_{1\le k\le N}\{ \underbrace{p_k^*(\beta)-\lambda}_{\tilde p_k(\beta)}\}\right\} \\ 
 &=\!\! \max_{0\leq k\leq N}\{\tilde p_k(\beta)\}.
	\end{align*}
	A direct method for characterizing $f^*$ is to identify the intersections of $\tilde p_k$ and $\tilde p_l$ for all possible pairs $0\leq k<l\leq N$, sort these intersections, and then determine the minimum piece within every pair of adjacent intersections. This method correctly characterizes $f^*$ and operates in $\mathcal{O}(N^2)$. However, we demonstrate that this complexity can be improved to $\mathcal{O}(N)$. To explain our method, we start by introducing the class of \textit{semi-consistent} functions.

\begin{definition}\label{def_nondec}
    A piece-wise quadratic function $g$ with breakpoints $-\infty=\tau_{0}<\tau_{1}<\dots<\tau_{N}=+\infty$ and pieces $p_{1}, \dots, p_{N}$ is called \textbf{semi-consistent} if it satisfies the following properties:
    \begin{itemize}
    \item $p_1, \dots, p_{N}$ are strongly convex;
        \item We have $p_k(\alpha)\leq \min\{p_{k-1}(\alpha),p_{k+1}(\alpha)\}$ for all $\alpha\in [\tau_{k-1},\tau_k]$ and $2\leq k\leq N-1$.
        \item For all $k\leq N$, the indexing function $I_{g_k}$ is non-decreasing, where $g_k: \bbbr\to\bbbr$ is defined as: 
        \begin{align}
            g_k(\alpha)=\begin{cases}
                g(\alpha)\qquad &\alpha\le \tau_k\\
                p_k(\alpha)&\alpha>\tau_k.
            \end{cases}
        \end{align}
    \end{itemize}
\end{definition}
The first property mirrors that of consistent functions. The second property is a local variant of the second property of the consistent functions: Within the local interval bounded by two adjacent breakpoints $\tau_{k-1}$ and $\tau_{k}$, the function $g$ is the minimum of the adjacent pieces $p_{k-1}$, $p_k$, and $p_{k+1}$.
Moreover, the function $g_k$ is obtained by restricting the function $g$ to its first $k$ pieces, with the final piece extended to $+\infty$. Indeed, $g_k$ is piece-wise quadratic with $k$ strongly convex pieces. However, it may not be consistent. It is also evident that $g_N =  g$. 

\begin{figure}
    \centering
    \includegraphics[scale=0.55]{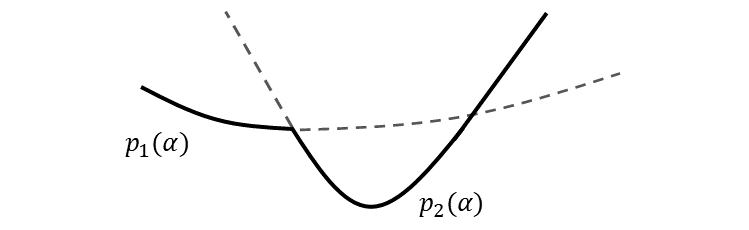}
    \caption{A semi-consistent function with two pieces. The function is not consistent since it violates the second property of Definition~\ref{def_consistent}.}
    \label{fig: Semiconsistent function}
\end{figure}
Not every semi-consistent function is consistent. An example is depicted in Figure~\ref{fig: Semiconsistent function}. However, our next lemma shows that every consistent function is semi-consistent.

\begin{lemma}\label{lemma: Indexing function for f_i }
    Any consistent quadratic function is semi-consistent.
\end{lemma}
\begin{proof}
Suppose that $g$ is consistent with $N$ pieces. The first property of semi-consistent functions is trivially satisfied for $g$. Since $g(\alpha) = \min_{1\leq k\leq N} \{p_k(\alpha)\}$, the function $g$ also satisfies the second property. To prove the last property, we can follow the same steps as the proof of Lemma~\ref{lem_consistent_non_decreasing}. The first step is to show that Lemma~\ref{lemma: addition prop of f for lemma1} holds for $g_k$. The second step is to prove the non-decreasing property of $I_{g_k}$ based on the statements of Lemma~\ref{lemma: addition prop of f for lemma1}. The details of the proof are omitted since they are identical to those of Lemma~\ref{lem_consistent_non_decreasing}. \qed
\end{proof}

Since every consistent function is semi-consistent, to prove Proposition~\ref{prop:g_efficient}, it suffices to provide an efficient algorithm for obtaining the conjugate of the functions expressed as $g+\lambda\bbbone_\alpha$, where $g$ is semi-consistent. 

	Recall the geometric interpretation of a conjugate function: Given any convex function $p_k$, the negative of its conjugate $-p_k^*$ is the intercept of a tangent to $p_k$ with slope $\beta$.
	\begin{definition}
		For any $1\leq k<l\leq N$, we define a \textbf{feasible common tangent} $s_{kl}$ to pieces $l$ and $k$ as the slope of a line that is tangent to  $p_k$ and $p_l$ at some points $\tau_{k-1}\leq \alpha_k\leq \tau_k$ and $\tau_{l-1}\leq \alpha_l\leq \tau_l$, respectively. 
	\end{definition}
Observe that since functions $p_k$ and $p_l$ are strictly convex, any tangent line is an underestimator of the function. Moreover, any two different lines in $\bbbr^2$ intersect in at most one point. If the intersection occurs in interval $[\tau_{k-1},\tau_k]$, then one line is strictly ``above" the other in interval $[\tau_{l-1},\tau_l]$ and they cannot both be tangents of $p_l$. Cases where the intersection occurs in a different interval or where the lines are parallel can be handled identically. We formally prove this result next.

 \begin{lemma}\label{lemma: Atmost 1 common tangent}
		For any $1\leq k<l\leq N$, the pieces $k$ and $l$ can have at most one feasible common tangent.
	\end{lemma}

	\begin{proof}
		Since a feasible common tangent $s_{kl}$ must satisfy $-p_l^*(s_{kl}) = -p_k^*(s_{kl})$ and both $p_{l}^*$ and $p_k^*$ are quadratic, the pieces $p_l$ and $p_k$ can have at most two feasible common tangents. By contradiction, suppose they have exactly two common tangents, given by $\ell^1_{kl}(\alpha):=s^1_{kl}\alpha+b^1_{kl}$ and $\ell^2_{kl}(\alpha):=s^2_{kl}\alpha+b^2_{kl}$. Without loss of generality, suppose $s^2_{kl}>s^1_{kl}$. Let $\alpha^1_{k}$ and $\alpha^2_{k}$ be the points at which the lines $\ell_{kl}^1(\alpha)$ and $\ell_{kl}^2(\alpha)$ are tangent to $p_k$, respectively. Define $\alpha^1_{l}$ and $\alpha^2_{l}$ in a similar fashion. Since $p_k$ and $p_l$ are strongly convex and $s^1_{kl}<s^2_{kl}$, we must have $\alpha^1_k<\alpha^2_k$ and $\alpha^1_l<\alpha^2_l$. Therefore,
		\begin{align}\label{eq_xkl}
			\tau_{k-1}\leq \alpha^1_k<\alpha^2_k\leq \tau_k\leq \tau_{l-1}\leq \alpha^1_l<\alpha^2_l\leq \tau_{l}\implies \alpha_k^2\leq \alpha_l^1.
		\end{align}
		On the other hand, due to the strong convexity of $p_k$, we have
		\begin{align*}
			\ell^1_{kl}(\alpha_k^1) &= p_k(\alpha_k^1)\quad \text{and}\quad	\ell^1_{kl}(\alpha) < p_k(\alpha); \forall \alpha\not=\alpha_k^1\\
			\ell^2_{kl}(\alpha_k^2) &= p_k(\alpha_k^2)\quad \text{and}\quad \ell^2_{kl}(\alpha) < p_k(\alpha); \forall \alpha\not=\alpha_k^2.
		\end{align*}
		Combining the above two inequalities, we have $\ell^1_{kl}(\alpha_k^2)<\ell^2_{kl}(\alpha_k^2)$, which implies that $b_{kl}^1-b_{kl}^2<(s_{kl}^2-s_{kl}^1)\alpha_k^2$. Similarly, one can show that $\ell^2_{kl}(\alpha_l^1)<\ell^1_{kl}(\alpha_l^1)$, which implies that $(s_{kl}^2-s_{kl}^1)\alpha_l^1<b_{kl}^1-b_{kl}^2$. Therefore, we have
		\begin{align*}
			\begin{cases}
				b_{kl}^1-b_{kl}^2<(s_{kl}^2-s_{kl}^1)\alpha_k^2\\
				(s_{kl}^2-s_{kl}^1)\alpha_l^1<b_{kl}^1-b_{kl}^2
			\end{cases}\implies \alpha_l^1<\alpha_k^2.
		\end{align*}
		This contradicts~\eqref{eq_xkl}, thereby completing the proof. \qed
	\end{proof}
	
Our next algorithm (Algorithm~\ref{alg: slope}) obtains the value of $s_{kl}$.
	\begin{algorithm}[h]
		\caption{Feasible common tangent: $\texttt{SLOPE}(p_{k}, \tau_{k-1}, \tau_{k}, p_{l}, \tau_{l-1}, \tau_{l})$}\label{alg: slope}
		\textbf{Input:} $\{p_{k}, \tau_{k-1}, \tau_{k}\}$ and $\{p_{l}, \tau_{l-1}, \tau_{l}\}$\\
		\textbf{Output:} The slope of the feasible common tangent $s_{kl}$
		\begin{algorithmic}[1]
			\State Obtain the conjugate functions $p_k^*$ and $p_l^*$
			\State Obtain the roots $\beta^1_{kl}$ and $\beta^2_{kl}$ of $-p_k^*(\beta)=-p_l^*(\beta)$
			\State Obtain $\alpha^1_k = \argmax_\alpha\{\beta^1_{kl}\alpha-p_k(\alpha)\}$ and $\alpha^1_l = \argmax_\alpha\{\beta^1_{kl}\alpha-p_l(\alpha)\}$
			\State Obtain $\alpha^2_k = \argmax_\alpha\{\beta^2_{kl}\alpha-p_k(\alpha)\}$ and $\alpha^2_l = \argmax_\alpha\{\beta^2_{kl}\alpha-p_l(\alpha)\}$
            \If{$\alpha^1_k\in [\tau_{k-1}, \tau_k]$ and $\alpha^1_l\in [\tau_{l-1}, \tau_l]$} \label{algline:slope if1}
			\State \Return $s_{kl} = \beta^1_{kl}$
			\ElsIf{$\alpha^2_k\in [\tau_{k-1}, \tau_k]$ and $\alpha^2_l\in [\tau_{l-1}, \tau_l]$}\label{algline:slope if2}
			\State \Return $s_{kl} = \beta^2_{kl}$
			\ElsIf{$\alpha^1_l\notin [\tau_{l-1},\tau_l] $ and $\alpha^2_l\notin[\tau_{l-1},\tau_l]$}\label{algline:slope if3}
			\State \Return $s_{kl} = +\infty$
            \Else \label{algline:slope if4}
            \State \Return $s_{kl} = -\infty$
			\EndIf
		\end{algorithmic}
	\end{algorithm}

	A few observations are in order regarding Algorithm~\ref{alg: slope}. First, note that the conjugate functions $p_k^*$ and $p_l^*$ in Line 1 can be obtained in $\mathcal{O}(1)$ time and memory. Moreover, without loss of generality, we assume that $-p_k^*(\beta)=-p_l^*(\beta)$ has two roots $\beta^1_{kl}$ and $\beta^2_{kl}$; indeed, the later steps of the algorithm can be modified accordingly if $-p_k^*(\beta)=-p_l^*(\beta)$ has fewer than two roots. It is also easy to see that $\{\alpha_k^1,\alpha_l^1, \alpha_k^2, \alpha_l^2\}$ in Lines 3 and 4 can be obtained in $\mathcal{O}(1)$ time and memory. Finally, the algorithm assigns $+\infty$ or $-\infty$ to $s_{kl}$ if a feasible common tangent does not exist.

We next show that the breakpoints of $g^*$ coincide with certain feasible common tangents that satisfy a breakpoint condition.

\begin{definition}
We say pieces $k<l$ satisfy the \textbf{breakpoint condition} if:
\begin{itemize}
    \item $-\infty< s_{kl}<+\infty$;
    \item $I_g^-(s_{kl}) = \lim_{\epsilon\to 0^+}I_g(s_{kl}-\epsilon)=I_g(s_{kl})=k$;
    \item $I_g^+(s_{kl}) = \lim_{\epsilon\to 0^+}I_g(s_{kl}+\epsilon)=l$.
\end{itemize}
\end{definition}

{We refer the reader back to Figure~\ref{fig: Indexing Function} for intuition. Both lines with slopes $\beta_1$ and $\beta_2$ are tangent to pieces satisfying the breakpoint condition. Alternatively, imagine the line tangent to pieces $p_2$ and $p_3$. Such a line would cut into the epigraph of piece $p_4$. In this scenario, $I_g^+(s_{23})=4$, violating the last condition. Intuitively, tangent lines between pieces satisfying the breakpoint condition are the lines required to describe the convex envelope of the piece-wise quadratic function $g$. More formally, as we show next, the slopes of such lines are required to describe the conjugate function.
}

\begin{lemma}\label{lemma: breakpoint condition}
    The pieces $k<l$ satisfy the breakpoint condition if and only if their feasible common tangent $s_{kl}$ is a breakpoint for $g^*$.
\end{lemma}
\begin{proof}
    Suppose that the pieces $k<l$ satisfy the breakpoint condition. Therefore, we have $I_g^-(s_{kl}) = I_g(s_{kl})=k$, which implies that there exists some $\overline{\epsilon}> 0$ such that for all $\epsilon\in [0,\overline{\epsilon})$  we have $I_g(s_{kl}-\epsilon)=k$. From the definition of the indexing function, it follows that there exists $\alpha^\star_k\in\argmax\{(s_{kl}-\epsilon)\alpha-g(\alpha)\}$ such that $\tau_{k-1}\le \alpha^\star_k \le \tau_k$. Therefore, we have
    \begin{align*}
        g^*(s_{kl}-\epsilon)&=(s_{kl}-\epsilon)\alpha^\star_k-g(\alpha^\star_k)\\
        &=(s_{kl}-\epsilon)\alpha_k^\star-p_k(\alpha_k^\star)\\
        &=\max_{\alpha}\{(s_{kl}-\epsilon)\alpha-p_{k}(\alpha)\}\\
        &=p^*_k(s_{kl}-\epsilon).
    \end{align*}
    Similarly, since $I_g^+(s_{kl})=l$, there exists some $\overline{\epsilon}>0$ such that for all $\epsilon\in (0,\overline{\epsilon})$ we have $g^*(s_{kl}+\epsilon)=p^*_l(s_{kl}+\epsilon)$.
    The above two equations imply that $s_{kl}$ is indeed a breakpoint of $g^*$.

    Conversely, suppose that a point $\tau$ is a breakpoint for $g^*$. Since $\tau$ is a breakpoint, we must have $I^-_g(\tau)\not=I^+_g(\tau)$. This together with the non-decreasing property of $I_g$ implies that $I^-_g(\tau)<I^+_g(\tau)$. Let $k  = I^-_g(\tau)$ and $l= I^+_g(\tau)$ for some $k<l$. We proceed to prove that $\tau$ is indeed the feasible common tangent to the pieces $k$ and $l$. First, it is easy to verify that $p_k$ and $p_l$ cannot be identical. Define $\alpha^\star_k = \argmax_\alpha\{\tau \alpha-p_k(\alpha)\}$ and $\alpha^\star_l = \argmax_\alpha\{\tau \alpha-p_l(\alpha)\}$. Due to the definition of the indexing function, we have $\tau_{k-1}\leq \alpha^\star_k\leq \tau_k$ and $\tau_{l-1}\leq \alpha^\star_l\leq \tau_l$. Consider the lines $\ell_{\tau,k}(\alpha) = \tau \alpha-p_k^*(\tau)$ and $\ell_{\tau,l}(\alpha) = \tau \alpha-p_l^*(\tau)$. Indeed, these two lines are tangent to pieces $k$ and $l$ at points $\alpha_k^\star$ and $\alpha_l^\star$, respectively. Moreover, they coincide since $p_k^*(\tau) = p_l^*(\tau)$ due to the continuity of $g^*$. Therefore, $\tau$ is the feasible common tangent to the pieces $k$ and $l$.\qed
\end{proof}

\begin{figure}
    \centering
    \includegraphics[scale=0.28]{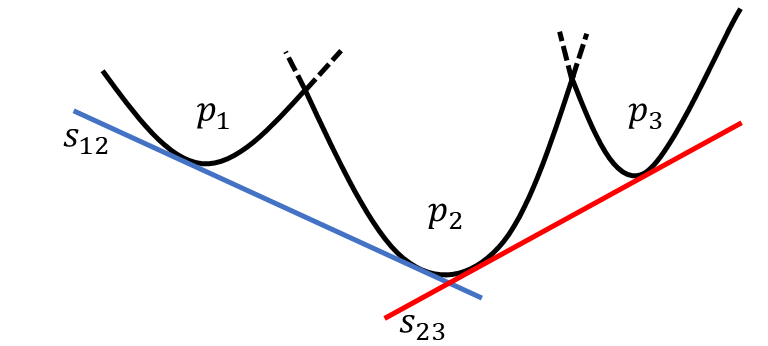}
    \qquad\includegraphics[scale=0.28]{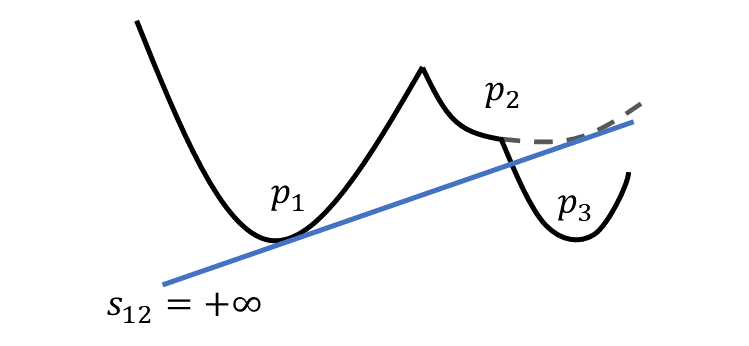}\\
    \includegraphics[scale=0.28]{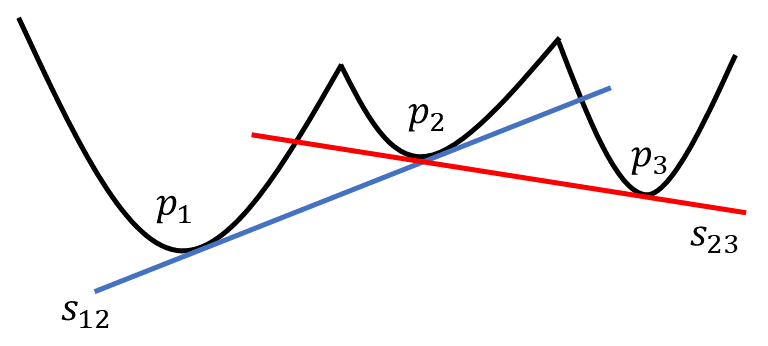}\qquad
    \includegraphics[scale=0.28]{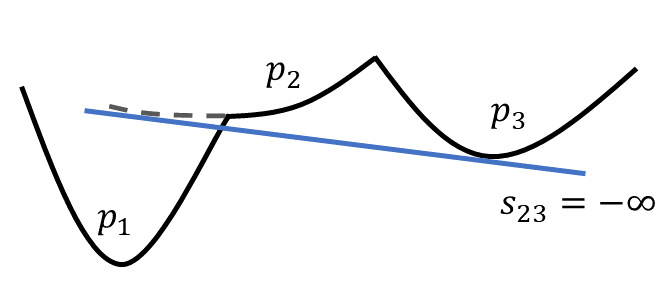}
    \caption{The first row corresponds to the \textbf{ADD} step of Algorithm~\ref{alg: breakpoint}. The second row corresponds to the \textbf{DELETE} step, wherein piece $p_2$ is discarded by the algorithm.}
    \label{fig:BP_figure}
\end{figure}

According to Lemma \ref{lemma: breakpoint condition}, it suffices to identify every pair of pieces $k<l$ that satisfy the breakpoint condition. {This can be naturally achieved by verifying the condition for all $n\choose 2$ pairs of pieces.} Our proposed Algorithm~\ref{alg: breakpoint}, which we call the \textit{breakpoint algorithm}, achieves this goal {in linear time}. It keeps track of two ordered lists $\Gamma$ and $\Pi$. The list $\Gamma$ collects the set of candidate breakpoints, whereas the list $\Pi$ records the pieces that satisfy the breakpoint condition. In other words, upon termination, the pieces $\Pi(j)$ and $\Pi(j+1)$ satisfy the breakpoint condition for any $j=1,\dots, |\Pi|-1$. The initial values of these lists are set as $\Gamma = [-\infty]$ and $\Pi = [1]$. 

\begin{algorithm}[]
		\caption{Breakpoint algorithm}\label{alg: breakpoint}
		\textbf{Input:} $g+\lambda\bbbone_\alpha$, where $g$ is semi-consistent\\
		\textbf{Output:} The conjugate of the input function 
		\begin{algorithmic}[1]
			\State $\Gamma \leftarrow [-\infty]$\Comment{Ordered list of candidate breakpoints of $g^*$}
			\State $\Pi \leftarrow [1]$\Comment{Ordered indices satisfying the breakpoint condition}
			\State $j\leftarrow 2$
			\While{$j\leq N$}
			\State $i\leftarrow \texttt{end}(\Pi)$\Comment{Return the last (maximum) element of $\Pi$}
			\State $s_{ij}\leftarrow \texttt{SLOPE}(p_{i}, \tau_{i-1}, \tau_{i}, p_{j}, \tau_{j-1}, \tau_{j})$\Comment{Obtain the feasible common tangent}
   \If{$s_{ij} > \texttt{end}(\Gamma)$}\Comment{{\bf ADD}} \label{algline:if}
			\State $\Gamma\leftarrow \texttt{append}(\Gamma,s_{ij})$\Comment{Append $s_{ij}$ to $\Gamma$ as a new breakpoint}
			\State $\Pi\leftarrow \texttt{append}(\Pi,j)$\Comment{Append $j$ to $\Pi$}
                \State $j\leftarrow j+1$
   \ElsIf{$s_{ij} \leq \texttt{end}(\Gamma)$}\Comment{{\bf DELETE}} \label{algline:elseif}
			\State $\Gamma\leftarrow \texttt{delete}(\Gamma,\texttt{end}(\Gamma))$\Comment{Delete the last breakpoint from $\Gamma$}
			\State $\Pi\leftarrow \texttt{delete}(\Pi,\texttt{end}(\Pi))$\Comment{Delete the last index from $\Pi$}
			\EndIf
			\EndWhile
			\State $\Gamma\leftarrow \texttt{append}(\Gamma,+\infty)$
			\State Define $g^*(\beta) = p_{\Pi(k)}^*(\beta),\ \text{for}\ \Gamma(k)\leq \beta\leq \Gamma(k+1) ;\ k=1,\ldots,M$. \label{algline:bpg*}
			\State Find the roots $\beta_1<\beta_2$ of $-g(0) = g^*(\beta)-\lambda$ \label{algline:bproots}
			\State \Return the conjugate of $g(\alpha)+\lambda\bbbone_\alpha$ as
			$$
			 \begin{cases}
				-g(0) & \beta_1\leq \beta\leq \beta_2\\
				g^*(\beta)-\lambda & \text{otherwise}
			\end{cases}
			$$ \label{alg:return}
		\end{algorithmic}
	\end{algorithm}
	
	\begin{figure}
		\centering
		\includegraphics[scale=0.35]{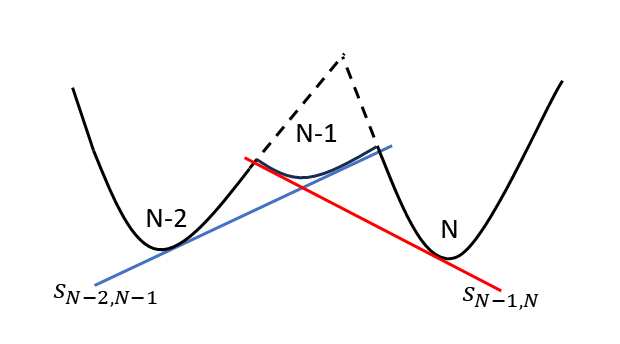}
		\caption{The auxiliary function $\tilde g$ defined by removing piece $N-1$ from $g$, and extending the pieces $N-2$ and $N$ to substitute piece $N-1$.}
		\label{fig: auxiliary}
	\end{figure}
 
At every iteration, the algorithm takes one of the following steps:
\begin{itemize}
    \item {\bf ADD (Line 7 of Algorithm~\ref{alg: breakpoint}):} When a common tangent between the piece $j$ and the highest index $i$ in $\Pi$ is greater than the largest discovered breakpoint in $\Gamma$, the algorithm adds the index $j$ and the common tangent $s_{ij}$ to the lists $\Pi$ and $\Gamma$, respectively. This scenario is depicted in the first row of Figure~\ref{fig:BP_figure}. Note that, at this step, it is possible for the algorithm to add an infeasible common tangent with $s_{ij}=+\infty$ to $\Gamma$ (see Figure~\ref{fig:BP_figure}, top right figure). However, both $s_{ij}$ and $j$ will be discarded in the {DELETE} step, as we explain next.
    \item {\bf DELETE (Line 11 of Algorithm~\ref{alg: breakpoint}):} When a common tangent between the piece $j$ and the piece with the highest index $i$ in $\Pi$ is smaller than the largest discovered breakpoint $\tau$ in $\Gamma$, the algorithm deletes the last elements of the lists $\Gamma$ and $\Pi$. Intuitively, this condition implies that the last piece of $\Pi$ cannot satisfy the breakpoint condition when paired with any other piece. As another interpretation, this piece does not play a role in characterizing the convex envelope of $g$ since it lies in the interior of its epigraph. This scenario is also depicted in the second row of Figure~\ref{fig:BP_figure}.
\end{itemize}

Our next theorem shows that the breakpoint algorithm returns the conjugate of any function $g(\alpha)+\lambda\bbbone_\alpha$, provided that $g$ is semi-consistent.
\begin{theorem}\label{thm_g}
    Let $g$ be semi-consistent with $N$ pieces. The breakpoint algorithm (Algorithm~\ref{alg: breakpoint}) correctly computes the conjugate of $g+\lambda\bbbone_{\alpha}$ for any $\lambda>0$ in $\mathcal{O}(N)$ time and memory.
\end{theorem}
Before presenting the proof of the above theorem, we show how it can be used to complete the proof of Proposition~\ref{prop:g_efficient}.

\paragraph{Proof of Proposition~\ref{prop:g_efficient}.}  According to Lemma~\ref{lemma: Indexing function for f_i }, $\tilde f$ is semi-consistent. Therefore, the proof readily follows upon choosing $g = \tilde f$ in Theorem~\ref{thm_g}.\qed\vspace{2mm}

Next, we present the main idea behind the correctness proof of the breakpoint algorithm. Our proof is based on induction on the number of pieces in $g$. Suppose the breakpoint algorithm returns the conjugate of any semi-consistent function with at most $N-1$ pieces. Our goal is to use this assumption to prove that the algorithm returns the conjugate of $g_N$ with $N$ pieces. Note that, when running the breakpoint algorithm on $g_N$, the algorithm first processes the first $N-1$ pieces of $g_N$, which are identical to $g_{N-1}$. Due to Definition~\ref{def_nondec}, $g_{N-1}$ is semi-consistent with $N-1$. 
Therefore, relying on our induction hypothesis, the breakpoint algorithm correctly identifies the breakpoints and pieces of $g^*_{N-1}$. Let $s_{i,N-1}$ and $N-1$ denote the last breakpoint and piece added to $\Gamma$ and $\Pi$ respectively until the algorithm reaches piece $N$. Upon processing piece $N$, two potential scenarios emerge:
\begin{itemize}
	\item \underline{Case 1:  $s_{i,N-1}<s_{N-1,N}$.} In this case, the algorithm ``adds'' the breakpoint $s_{N-1,N}$ and the piece $N$ to $\Gamma$ and $\Pi$, then returns these sets as the set of breakpoints and pieces of $g^*_N$. We prove that these sets coincide with the true sets of breakpoints and pieces of $g^*_N$.
	\item \underline{Case 2: $s_{i,N-1}\geq s_{N-1,N}$.} In this scenario, the algorithm ``deletes'' the breakpoint $s_{i,N-1}$ and the piece $N-1$ from $\Gamma$ and $\Pi$ respectively. Here, we establish that the piece $N-1$ does not contribute to the characterization of $g_N^*$. In this scenario, $g^*_N$ is the same as the conjugate of an auxiliary function $\tilde g_{N-1}$, obtained by removing piece $N-1$ from $g_N$, and subsequently, extending pieces $N-2$ and $N$ to substitute piece $N-1$. Figure~\ref{fig: auxiliary} illustrates this function. We show that the constructed $\tilde g_{N-1}$ is semi-consistent and has $N-1$ pieces. Therefore, by induction hypothesis, the algorithm correctly recovers its conjugate.
\end{itemize}

The rest of this section is devoted to formalizing the above intuition. 

\paragraph{Proof of Theorem~\ref{thm_g}.}
We begin by presenting the proof of correctness, followed by the proof of its runtime.  Suppose that Line~\ref{algline:bpg*} correctly recovers $g^*$. Upon finding the roots $\beta_1<\beta_2$ of $-g(0) = g^*(\beta)-\lambda$, Equation~\eqref{eq_fstar} can be invoked to show that Line~\ref{alg:return} returns the conjugate of $g+\lambda\bbbone_{\alpha}$. Therefore, to prove the correctness of the algorithm, it suffices to show that Line~\ref{algline:bpg*} correctly recovers $g^*$. To this goal, we prove that the ordered lists $\Gamma$ and $\Pi$ coincide with the correct breakpoints and pieces of $g^*$, respectively. Our proof is by induction on the number of pieces $N$ of $g$. Recall that $g = g_N$ as defined Definition~\ref{def_nondec}.
To streamline the presentation, we keep the dependency of $g$ on $N$ explicit throughout the proof.

\paragraph{Base case.} Suppose $N=1$. Indeed, both $g_N$ and $g_N^*$ have one piece with no breakpoints. Since the While loop in Line 4 starts only when $g$ has more than one piece, the algorithm correctly returns the initial values of $\Pi=[1]$ and $\Gamma=[-\infty,+\infty]$. Thus, the base case of the induction hypothesis is true.

\paragraph{Induction step.} Suppose that the breakpoint algorithm correctly recovers $\Gamma$ and $\Pi$ for any semi-consistent function $g_{N-1}$ with at most $N-1$ pieces. Our goal is to prove that the algorithm correctly recovers the correct breakpoints and pieces for any semi-consistent function $g_N$ with $N$ pieces.

\begin{sloppypar}
    We use $(\Gamma^\star_N, \Pi^\star_N)$ and $(\Gamma_N, \Pi_N)$ to denote the true set of breakpoints and pieces of $g^*_N$, and those returned by the algorithm, respectively. Similarly, $({\Gamma}^\star_{N-1}, {\Pi}^\star_{N-1})$ and $({\Gamma}_{N-1}, {\Pi}_{N-1})$ are the true breakpoints and pieces, and those returned by the algorithm for $g^*_{N-1}$, respectively. From our induction hypothesis, we have ${\Gamma}^\star_{N-1}={\Gamma}_{N-1}$ and ${\Pi}^\star_{N-1}={\Pi}_{N-1}$. When we apply the algorithm to $g_N$, the algorithm first processes the first $N-1$ pieces of $g_N$. Let $(\tilde{\Gamma}_{N-1},\tilde\Pi_{N-1})$ denote the set of breakpoints and pieces returned by the algorithm at this point. 
For $g_{N-1}$, the piece ${N-1}$ is defined over the domain $[\tau_{N-2},\infty)$. Therefore, we have $\lim\limits_{\beta\to \infty}I_{g_{N-1}}(\beta)=N-1$. Suppose $i$ is the piece for which the pair $i$ and $N-1$ satisfies the breakpoint condition for $g_{N-1}$. This implies that
\end{sloppypar}
\begin{enumerate}
	\item $I_{g_{N-1}}^-(s_{i,N-1})=I_{g_{N-1}}(s_{i,N-1})=i$,
	\item $I_{g_{N-1}}^+(s_{i,N-1})=N-1$.
\end{enumerate}
Due to the non-decreasing property of $I_{g_{N-1}}$, we have $s_{i,N-1}=\max\{\Gamma^\star_{N-1}\}$. We consider two cases:

\paragraph{\underline{Case 1:  $s_{i,N-1}< s_{N-1,N}$.}} In this case, the algorithm proceeds with the ADD step and returns ${\Gamma_N}=\tilde \Gamma_{N-1}\cup\{s_{N-1,N}\}$ and $\Pi_N=\tilde\Pi_{N-1}\cup\{N\}$. We show that these sets coincide with $({\Gamma^\star_N},\Pi_N^\star)$.

\begin{Claim}
	$\tilde{\Gamma}_{N-1}=\Gamma^\star_{N-1}$ and $\tilde{\Pi}_{N-1}=\Pi^\star_{N-1}$.
\end{Claim}

To prove this claim, we first observe that the algorithm runs identically over the first $N-2$ pieces of $g_{N-1}$ and $g$, since these functions are identical over $(-\infty,\tau_{N-1}]$. Therefore, it follows that $\tilde{\Gamma}_{N-1}$ matches $\Gamma^\star_{N-1}$ entirely, except for a potential distinction in their final elements. This distinction occurs only if $s_{i,N-1}=-\infty$ or $s_{i,N-1}=+\infty$. Since $s_{i,N-1}\in \Gamma_{N-1}^\star$, we have $-\infty<s_{i,N-1}$. Moreover, since $\tau_N=+\infty$, we have $s_{N-1,N}<+\infty$ according to Algorithm~\ref{alg: slope}. This implies that $-\infty<s_{i,N-1}<s_{N-1,N}<+\infty$.  Therefore, both $s_{i,N-1}$ and $s_{N-1,N}$ are finite and $\tilde{\Gamma}_{N-1}=\Gamma^\star_{N-1}$. The proof of $\tilde{\Pi}_{N-1}=\Pi^\star_{N-1}$ follows similarly.

\paragraph{} Based on the above claim, it suffices to show that $\Gamma^\star_N = \Gamma^\star_{N-1}\cup \{s_{N-1,N}\}$ and $\Pi^\star_N = \Pi^\star_{N-1}\cup \{N\}$. To this goal, we rely on two crucial claims.
\begin{Claim}\label{claim:g_N_limited}
	$g^\star_{N-1}(\beta)=\max\limits_{\alpha\le \tau_{N-1}}\left\{ \alpha\beta-g_{N-1}(\alpha)\right\}$  for every $\beta<s_{N-1,N}$.
\end{Claim}
To prove the above claim, it suffices to show that, for every $\beta<s_{N-1,N}$, there exists some $\alpha^\star(\beta)\in\{\argmax_{\alpha}\{\alpha\beta-g_{N-1}(\alpha)\}\}$ such that $\alpha^\star(\beta)<\tau_{N-1}$. First consider the case $\beta\le s_{i,N-1}$. In this case, $I_{g_{N-1}}(\beta)\leq i$, which in turn implies $\alpha^\star(\beta)<\tau_{i}<\tau_{N-1} $. When $s_{i,N-1}<\beta\leq s_{N-1,N}$, from the non-decreasing property of the indexing function, we have $I_{g_{N-1}}(\beta)= N-1$. Thus, $\max_{\alpha}\{\alpha\beta-g_{N-1}(\alpha)\}=\max_{\alpha}\{\alpha\beta-p_{N-1}(\alpha)\}$ for every $s_{i,N-1}<\beta\leq s_{N-1,N}$. Since $p_{N-1}$ is strongly convex, $\alpha^\star(\beta)$ is an increasing function of $\beta$ for every $s_{i,N-1}<\beta\leq s_{N-1,N}$. On the other hand, $\alpha^\star(s_{N-1,N}) = \argmax_\alpha\{\alpha s_{N-1,N}-p_{N-1}(\alpha)\}<\tau_{N-1}$, where the last inequality follows from the fact that $s_{N-1,N}$ is finite and is the feasible common tangent to pieces $N-1$ and $N$. Therefore, we have $\alpha^\star(\beta)<\tau_{N-1}$ for every $\beta<s_{N-1,N}$.

\begin{Claim}\label{claim:g_N_p_N}
$g^*_{N-1}(\beta)>\alpha\beta-p_{N}(\alpha)$ for every $\beta<s_{N-1,N}$ and $\alpha>\tau_{N-1}$.
\end{Claim}
To prove this claim, define the line $\ell_{\beta}(\alpha)= \alpha\beta-g^*_{N-1}(\beta)$. It is easy to see that
\begin{align}\label{eq_claim3}
	\ell_{s_{N-1,N}}(\alpha) > \ell_{\beta}(\alpha),\qquad \text{for every $\beta<s_{N-1,N}$ and $\alpha>\tau_{N-1}$.}
\end{align}
Since $I_{g_{N-1}}(s_{N-1,N})=N-1$, it follows that $g^*_{N-1}(s_{N-1,N})=p^*_{N-1}(s_{N-1,N}) = p^*_N(s_{N-1,N})$. Thus,
\begin{align*}
	\ell_{s_{N-1,N}}(\alpha) &= s_{N-1,N}\alpha-g^*_{N-1}(s_{N-1,N})=s_{N-1,N}\alpha-p^*_{N}(s_{N-1,N}).
\end{align*}
On the other hand, due to the property of conjugate functions, for all $\alpha\in \bbbr$,
\begin{align*}
	p_N(\alpha)&\ge s_{N-1,N}\alpha-p^*_{N}(s_{N-1,N})= \ell_{s_{N-1,N}}(\alpha).
\end{align*}
The above inequality together with~\eqref{eq_claim3} implies that
\begin{align*}
	&p_N(\alpha)>\ell_{\beta}(\alpha),\qquad\qquad\quad\ \ \text{for every $\beta<s_{N-1,N}$ and $\alpha>\tau_{N-1}$}\\
	\iff &p_N(\alpha)> \beta\alpha-g^*_{N-1}(\beta),\qquad \text{for every $\beta<s_{N-1,N}$ and $\alpha>\tau_{N-1}$}\\
	\iff &g^*_{N-1}(\beta)>\beta\alpha-p_N(\alpha),\qquad \text{for every $\beta<s_{N-1,N}$ and $\alpha>\tau_{N-1}$},
\end{align*}
which completes the proof of this claim.

\paragraph{}
Equipped with Claims~\ref{claim:g_N_limited} and~\ref{claim:g_N_p_N}, we are ready to prove $\Gamma^\star_N = \Gamma^\star_{N-1}\cup \{s_{N-1,N}\}$ and $\Pi^\star_N = \Pi^\star_{N-1}\cup \{N\}$. To this goal, it suffices to show that 
\begin{align*}
	g^*_N(\beta)=\begin{cases}
		g^*_{N-1}(\beta)\qquad&\beta<s_{N-1,N}\\
		p^*_N(\beta)& \beta\ge s_{N-1,N}.
	\end{cases}
\end{align*}
Consider the case $\beta<s_{N-1,N}$. From the definition of the conjugate function, we have 
\begin{align*}
	g^\star_{N}(\beta)&=\max_{\alpha\in\bbbr}\left\{ \alpha\beta-g_{N}(\alpha)\right\}\\
	&=\max\left\{\max_{\alpha\le \tau_{N-1}}\left\{ \alpha\beta-g_{N}(\alpha)\right\},\max_{\alpha>\tau_{N-1}}\left\{ \alpha\beta-g_{N}(\alpha)\right\}\right\}\\
	&=\max\left\{\max_{\alpha\le \tau_{N-1}}\left\{ \alpha\beta-g_{N-1}(\alpha)\right\},\max_{\alpha>\tau_{N-1}}\left\{ \alpha\beta-p_{N}(\alpha)\right\}\right\}\\
	&=\max\left\{g_{N-1}^*(\beta),\max_{\alpha>\tau_{N-1}}\left\{ \alpha\beta-p_{N}(\alpha)\right\}\right\} \\
	&=g_{N-1}^*(\beta),
\end{align*}
where the second to last equality follows from Claim~\ref{claim:g_N_limited} and the last equality is due to Claim~\ref{claim:g_N_p_N}. Using the fact that $g^*_{N}(\beta)=g^*_{N-1}(\beta)$ for $\beta<s_{N-1,N}$, we obtain $I_{g_{N}}(\beta)=I_{g_{N-1}}(\beta)=N-1$ for $s_{i,N-1}<\beta<s_{N-1,N}$. On the other hand, $\lim_{\beta\to+\infty}I_g(\beta)=N$. Therefore, $s_{N-1,N}\in \Gamma^\star_N$, which implies $g^\star_{N}(\beta)=p^*_N(\beta)$ for $\beta\ge s_{N-1,N}$. This completes the proof of the first case.

\paragraph{\underline{Case 2:  $s_{i,N-1}\geq s_{N-1,N}$.}} In this case, the algorithm proceeds with the DELETE step and discards $s_{i,N-1}$ and $N-1$ from $\tilde\Gamma_{N-1}$ and $\tilde\Pi_{N-1}$, respectively. Our next claim shows that both $s_{i,N-1}$ and $N-1$ are correctly deleted, as piece $N-1$ does not belong to $\range(I_{g_N})$. 
\begin{Claim}\label{claim:Image}
	$N-1\not\in\range(I_{g_N})$.
\end{Claim}
To prove this claim, suppose, by contradiction, that $N-1\in \range(I_{g_N})$. This implies that there exists a piece $k$ such that both pairs $k,N-1$ and $N-1,N$ satisfy the breakpoint condition for $g_N$. Therefore, $s_{k,N-1}<s_{N-1,N}\leq s_{i,N-1}$. Due to the non-decreasing property of $I_{g_N}$, we have $I_{g_N}(\beta)\leq N-1$ for every $\beta\leq s_{N-1,N}$. This implies the existence of $\alpha^\star(\beta)\leq \tau_{N-1}$ such that $\alpha^\star(\beta)\in\argmax_{\alpha} \{\alpha \beta-g_{N-1}(\alpha)\}$ for every $\beta\leq s_{N-1,N}$. Therefore, we have $g^*_{N-1}(\beta)=\max_{\alpha} \{\alpha \beta-g_{N-1}(\alpha)\}=\max_{\alpha\le \tau_{N-1}} \{\alpha \beta-g_{N-1}(\alpha)\}$ for every $\beta\leq s_{N-1,N}$. Similarly, we have $I_{g_{N-1}}(\beta)\leq i$ for every $\beta\leq s_{i,N-1}$, which leads to $g^*_{N}(\beta)=\max_{\alpha\le \tau_{N-1}} \{\alpha \beta-g_{N}(\alpha)\}$ for every $\beta\leq s_{i,N-1}$. Combining these two equalities, for every $\beta\leq s_{N-1,N}\leq s_{i,N-1}$, we have
\begin{align*}
	g^*_{N-1}(\beta)&=\max_{\alpha\le \tau_{N-1}} \{\alpha \beta-g_{N-1}(\alpha)\}=\max_{\alpha\le \tau_{N-1}} \{\alpha \beta-g_{N}(\alpha)\} = g_N^*(\beta).
\end{align*}
The above equality implies that $I_{g_{N-1}}(\beta) = I_{g_{N}}(\beta) = N-1$ for every $s_{k,N-1}< \beta\leq s_{N-1,N}$. On the other hand, recall that $I_{g_{N-1}}(\beta)\leq i$ for every $\beta\leq s_{i,N-1}$, which leads to $I_{g_{N-1}}(\beta) <N-1$ for every $s_{k,N-1}< \beta\leq s_{N-1,N}$. This leads to a contradiction, thereby proving the claim. 

\paragraph{}
As the last step of the proof, we consider the following function:
\begin{align}
	\tilde{g}_{N-1}(\alpha)=\begin{cases}
		\min\{p_{N-2}(\alpha),p_{N}(\alpha)\}\qquad& \tau_{N-2}<\alpha<\tau_{N-1},\\
		g(\alpha)&\text{otherwise}.
	\end{cases}
\end{align}
The function $\tilde{g}_{N-1}$ is obtained by removing piece $N-1$ from $g$, and subsequently, extending $p_{N-2}$ and $p_{N}$ to substitute piece $N-1$. Our final claim shows that $\tilde{g}_{N-1}$ and $g_N$ have the same conjugates.

\begin{Claim}
	$\tilde{g}^*_{N-1}=g^*_N$.
\end{Claim}
To prove this claim, note that $\tilde{g}_{N-1}(\alpha)$ and ${g}(\alpha)$ are identical except within the interval $[\tau_{N-2}, \tau_{N-1}]$. Due to the second property of semi-consistent functions (Definition~\ref{def_nondec}), we have $\tilde{g}_{N-1}(\alpha)\geq g(\alpha)$ within the interval $[\tau_{N-2}, \tau_{N-1}]$. This implies that
\begin{align*}
	& \beta\alpha-g(\alpha)\geq \beta\alpha-\tilde g_{N-1}(\alpha); \qquad \forall \alpha,\beta\in\bbbr\\
	\implies & \max_\alpha\{\beta\alpha-g(\alpha)\} \geq \max_\alpha\{\beta\alpha-\tilde g_{N-1}(\alpha)\}; \qquad \forall \beta\in\bbbr\\
	\implies & g^*(\beta)\geq \tilde g_{N-1}^*(\beta); \qquad \forall \beta\in\bbbr.
\end{align*}
On the other hand, due to Claim~\ref{claim:Image}, we have $N-1\notin \range(I_{g_N})$. Therefore, for every $\beta\in \bbbr$, there exists $\alpha^\star(\beta)\notin [\tau_{k-1},\tau_k]$ such that $\alpha^\star(\beta)\in \argmax_\alpha \{\beta\alpha-g(\alpha)\}$. This implies that, for every $\beta\in \bbbr$:
\begin{align*}
	g^*(\beta) &= \max_\alpha\{\beta\alpha-g(\alpha)\} \\
	&= \beta\alpha^\star(\beta)-g(\alpha^\star(\beta)) \\
	&= \beta\alpha^\star(\beta)-\tilde g_{N-1}(\alpha^\star(\beta)) \\
	&\leq \max_\alpha\{\beta\alpha-\tilde g_{N-1}(\alpha)\} = \tilde g^*_{N-1}(\beta).
\end{align*}
Combining the above two inequalities implies that $\tilde{g}^*_{N-1}(\beta)=g_N^*(\beta)$, thereby proving the claim.

\paragraph{}
After discarding piece $N-1$, the algorithm operates identically on $\tilde{g}_{N-1}$ and $g_N$. Indeed, $\tilde{g}_{N-1}$ is semi-consistent since it satisfies the properties outlined in Definition~\ref{def_nondec}. Given that $\tilde{g}_{N-1}$ contains $N-1$ pieces, by our induction hypothesis, the breakpoint algorithm correctly identifies the breakpoints and pieces of $\tilde{g}^*_{N-1}$, which coincide with those of $g_N^*$ as asserted in the above claim. This completes the correctness proof of the algorithm.

Finally, we analyze the runtime of the algorithm. 
We consider the operations within the While loop of Algorithm~\ref{alg: breakpoint}. Every execution of Algorithm~\ref{alg: slope} can be completed in $\mathcal{O}(1)$. To see this, note that the If conditions in {Lines \ref{algline:slope if1}, \ref{algline:slope if2}, \ref{algline:slope if3}, and \ref{algline:slope if4}} of Algorithm~\ref{alg: slope} can be checked in $\mathcal{O}(1)$ time. The remaining operations of the While loop either add or delete an element to a list, each taking $\mathcal{O}(1)$ time and memory. Thus a single round of the While loop can be executed in $\mathcal{O}(1)$ time and memory. 
Next, we show that the While loop executes at most $\mathcal{O}(N)$ rounds. Once a piece is deleted, it will never be revisited. Since at most $N$ pieces can be added and at most $N$ pieces can be deleted, the While loop can execute at most $\mathcal{O}(N)$ rounds.  
Finally, note that, since $\Pi$ and $\Gamma$ have $\mathcal{O}(N)$ elements, computing $g^*_N$ in Line \ref{algline:bpg*} requires $\mathcal{O}(N)$ time and memory. Similarly, it follows that Line \ref{algline:bproots} can be computed in $\mathcal{O}(N)$. Consequently, we conclude that Algorithm~\ref{alg: breakpoint} operates in $\mathcal{O}(N)$ time and memory.\qed

\section{Practical consideration}\label{sec:implementation}
The breakpoint algorithm (Algorithm \ref{alg: breakpoint}) is prone to numerical instabilities for trees with a large number of nodes. In this section, we explain the root cause of this numerical issue and describe a correction step that averts this without any compromises to the performance and accuracy of the algorithm.

Consider an arbitrary pair of nodes $u,v$ where $v=\child(u)$ and $v$ is not a branch. 
Since $f_u(\alpha)$ is consistent, it can be written as
	$$
	f_u(\alpha) = \min_{1\leq k\leq N_u}\{p_{u,k}(\alpha)\}+\lambda_u\bbbone_{\alpha},
	$$
	where $\{p_{u,k}(\alpha)\}_{k=1}^{N_u}$ are strongly convex and quadratic. For every $k=1,\dots, N_u$, let $p_{u,k}(\alpha) = \gamma_{u,k,1}\alpha^2+\gamma_{u,k,2}\alpha+\gamma_{u,k,3}$. Lemma~\ref{lem:f_path2} and Equation~\eqref{eq:recursion} imply that
	$$
	f_v(\alpha) = \min_{1\leq k\leq N_v}\{p_{v,k}(\alpha)\}+\lambda_v\bbbone_{\alpha},
	$$
	where
 \begin{align}\label{eq_quad_update}
     p_{v,k}(\alpha) &= \frac{1}{2}\alpha^2+c_v \alpha-p^*_{u,k}(-Q_{u,v}\alpha)\\
     &= \underbrace{\left(\frac{1}{2}-\frac{Q_{u,v}^2}{4\gamma_{u,k,1}}\right)}_{:=\gamma_{v,k,1}}\alpha^2+\underbrace{\left(c_v-\frac{\gamma_{u,k,2}Q_{u,v}}{2\gamma_{u,k,1}}\right)}_{:=\gamma_{v,k,2}}\alpha+\underbrace{\left(\gamma_{u,k,3}-\frac{\gamma_{u,k,2}^2}{4\gamma_{u,k,1}}\right)}_{:=\gamma_{v,k,3}}.
 \end{align}
 Suppose, for some arbitrary indices $k,l$, we obtain $\gamma_{v,k,1},\gamma_{v,l,1}$ from $\gamma_{u,k,1},\gamma_{u,l,1}$ based on the equation above. Taking $\epsilon=\left|\gamma_{u,k,1}-\gamma_{u,l,1}\right|$, we obtain
\begin{align*}
\left|\gamma_{v,k,1}-\gamma_{v,l,1}\right|&=\underbrace{\left|\dfrac{Q^2_{u,v}}{4\gamma_{u,k,1}\gamma_{u,l,1}}\right|}_{\rho}\cdot\epsilon.
\end{align*}
When $\gamma_{u,k,1},\gamma_{u,l,1}>Q_{u,v}/2$, we observe that $\rho<1$, resulting in a decrease in the discrepancy of the quadratic terms. This scenario is likely to occur in practice, as $|Q_{u,v}|< 1$ due to the positive definiteness of $Q$, and the quadratic coefficients remain close to $1/2$ due to~\eqref{eq_quad_update}. The shrinking effect of the update rule is exacerbated in situations where multiple neighboring nodes satisfy $\rho<1$, thereby leading to fast decay in $\epsilon$. As $\epsilon$ approaches machine precision, the breakpoint algorithm would suffer from numerical instabilities.

To address this challenge, we note that, since the slope of the common tangent $s_{kl}$ is proportional to $\epsilon^{-1}$, such errors arise only at breakpoints with significantly large absolute values. Our subsequent lemma demonstrates that these breakpoints correspond to suboptimal pieces, and thus can be easily excluded from consideration.

\begin{lemma}\label{lemma: M}
	Let $x^\star$ be the optimal solution of Problem \eqref{eq: MIQP}. We have $\|x^\star\|_\infty\le \frac{\|c\|_2}{\lambda_{\min}(Q)}$, where $\lambda_{\min}(Q)$ denotes the smallest eigenvalue of $Q$.
\end{lemma}
\begin{proof}
	Suppose $J$ corresponds to the set of row indices over which $x^\star$ is non-zero. We have $x^\star=-Q^{-1}_{J,J}c_{J}$, which implies
	\begin{align*}
 \|{x^\star}\|_{\infty}\leq 
\|{x^\star}\|_{2}&=\left\|(Q_{J,J})^{-1}c_{J}\right\|_{2}\le \|(Q_{J,J})^{-1}\|_{2}\|c\|_{2}\le \dfrac{\|c\|_2}{\lambda_{\min}(Q_{J,J})}.
	\end{align*}
Since $J\subset N$, we have
\begin{align*}
	\lambda_{\min}(Q_{J,J})&=\min_{\|x\|_2=1}x^\top Q_{J,J}x=\min_{\substack{\|x\|_2=1,\\x_{J}=0}}x^\top Qx\ge \min_{\|x\|_2=1}x^\top Qx\ge \lambda_{\min}(Q).
\end{align*}
This completes the proof.
\qed \end{proof}
According to the above lemma, it suffices to characterize the parametric cost at any node $u$ within the range $\left[-\frac{\|c\|_2}{\lambda_{\min}(Q)}, \frac{\|c\|_2}{\lambda_{\min}(Q)}\right]$. Therefore, the aforementioned numerical issue can be mitigated by first obtaining $\frac{\|c\|_2}{\lambda_{\min}(Q)}$ and then discarding the breakpoints falling outside the range $\left[-\frac{\|c\|_2}{\lambda_{\min}(Q)}, \frac{\|c\|_2}{\lambda_{\min}(Q)}\right]$.

\section{Experiments}\label{sec:experiments}
In this section, we assess the performance of our algorithm across various synthetic and real-world case studies. All experiments were run on a computer with 16 cores of 3.0 GHz Xeon Gold 6154 processors and 8 GB memory per core. Specifically, we compare the proposed parametric algorithm with Gurobi v10.0.2. For Gurobi, a time limit of 1 hour was set, and the algorithm was terminated whenever the optimality gap fell below $0.01\%$. If Gurobi failed to achieve an optimality gap of $0.01\%$ or less within this time limit, we reported the best optimality gap attained. We also note that Gurobi, from version 10 onwards, uses a branch-and-bound method based on a perspective reformulation to solve Problem~\eqref{eq: MIQP}; these reformulations are known to outperform the classical \textit{big-M} reformulations (see, e.g., \cite{xie2020scalable}) and are considered state-of-the-art. The Python implementation of our algorithm as well as the presented case studies are available at
\begin{center}
    \url{https://github.com/aareshfb/Tree-Parametric-Algorithm.git}
\end{center}

\subsection{Case Study on synthetic dataset}
For our first set of experiments, we construct $\supp(Q)$ as a randomly generated connected tree. The nonzero off-diagonal elements are selected from a uniform distribution within the range $[-1,0]$. Each diagonal element $Q_{i,i}$ is set to $1+\sum_{j\not=i}|Q_{i,j}|$. This ensures that $Q$ is positive definite. Similarly, elements of vector $c$ were generated from a uniform distribution within the interval $(-10,10)$. Unless explicitly stated otherwise, the default regularizing parameter was set to $\lambda_i=7.5$ for all $i$. This value approximately corresponds to 50\% non-zero elements in the optimal solution for the selected $Q$ and $c$.

First, we examine the performance of the parametric algorithm for problems with varying size $n$. The results are presented in Table \ref{table: Vary n}.

\begin{table}[h!]
	\begin{center}
		\caption{Performance for varying sizes}
		\label{table: Vary n}
		\begin{tabular}{ c|c|c|c|c|c|c } 
			\hline
			\textbf{Metric} & \textbf{Method} & $\pmb{n=200}$ & $\pmb{n=500}$&  $\pmb{n=1000}$& $\pmb{n=2000}$&$\pmb{n=5000}$\\ 
			\hline
			\multirow{2}{*}{Time(s)} & Parametric & 0.18& 0.48& 1.01& 2.16& 5.8\\ 
			& Gurobi & 123.93& TL& TL& TL &TL\\ 
			\hline
			B\&B nodes&\multirow{2}{*}{Gurobi}&1149920&1661958&6186925&2384160&789477\\
			Opt. gap && $\le 0.00\%$&1.25\%&1.48\%& 2.00\%& 2.15\%\\
			\hline
		\end{tabular}
	\end{center}
	{\footnotesize TL: Time Limit (1 hour). The reported results are averaged over 5 trials. ``Parametric'' refers to the parametric algorithm (Algorithm~\ref{alg: general trees}).}
\end{table}

	It can be seen that Gurobi is unable to solve instances with sizes exceeding $n=200$ within 1 hour. In contrast, our proposed parametric algorithm can solve instances with $n=5,000$ in less than 6 seconds, significantly outperforming Gurobi. To provide further insight into the efficiency of the parametric algorithm, we plot its runtime across a broader range of $n$ in Figure~\ref{fig: Exp1_complexity}. Notably, the parametric algorithm can solve instances of size $n=50,000$ within 2 minutes.

 \begin{figure}[htbp]
	\begin{center}
		\includegraphics[scale=0.5]{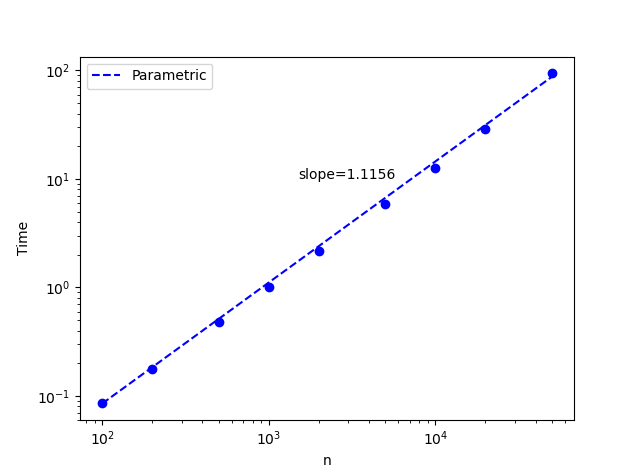}
	\end{center}
	\caption{The runtime of the parametric algorithm (Algorithm~\ref{alg: general trees}) for different values of $n$. The reported results are averaged over 5 trials.}
	\label{fig: Exp1_complexity}
\end{figure}
	
	Moreover, while the theoretical complexity of the parametric algorithm can be as high as $\mathcal{O}(n^2)$, in practice, we observe a complexity that is closer to linear $\mathcal{O}(n^{1.1156})$. This improved complexity can be attributed to the fact that, while the parametric cost at the root node $f_1(\alpha)$ may have up to $2n$ pieces, in practice, the number of pieces is expected to be significantly smaller. More specifically, recall that $N_u$ denotes the number of pieces in the parametric cost $f_u(x)$. We have shown that the runtime of Algorithm~\ref{alg: general trees} is $\mathcal{O}\left(\sum_{u=1}^n N_u\right) = \mathcal{O}\left(n\bar{N}\right)$, where $\bar{N}$ denotes the average number of pieces. While this leads to a quadratic runtime when $\bar{N}=\mathcal{O}(n)$, it becomes linear if $\bar{N}=\mathcal{O}(1)$.
	
	\begin{figure}[htbp]
		\begin{center}
			\includegraphics[scale=0.5]{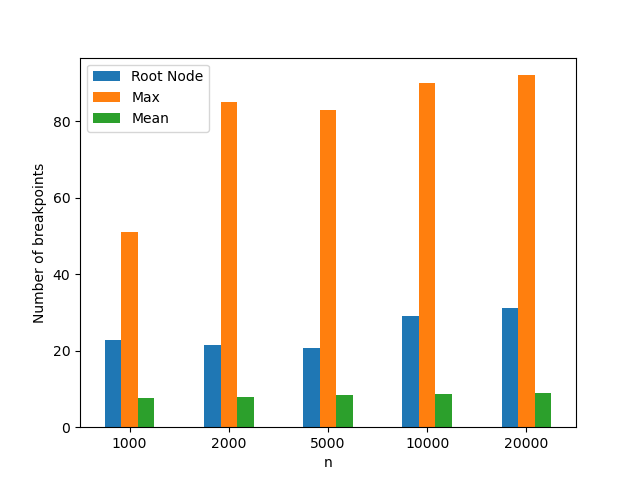}
		\end{center}
		\caption{The values of $\bar{N}$ (denoted as ``Mean''), $\max_{u}\{N_u\}$ (denoted as ``Max''), and $N_1$ (denoted as ``Root Node'') for different values of $n$. Note that the $\max_{u}\{N_u\}$ does not necessarily coincide with $N_1$. The reported results are averaged over 10 trials.}
		\label{fig: Exp1_no_bp}
	\end{figure}
	
	Figure \ref{fig: Exp1_no_bp} illustrates the average number of pieces generated by the parametric algorithm for different values of $n$. It is evident that as $n$ increases from $1,000$ to $20,000$, the average number of pieces ranges from $20$ to $35$. This observation supports our hypothesis that, in practice, the average number of pieces grows only sublinearly with respect to $n$.

Next, we fix $n=1,000$ and compare the performance of the parametric algorithm and Gurobi for different regularization parameters $\lambda$. Specifically, we set $\lambda_1=\dots=\lambda_n=\bar\lambda$ and vary $\bar\lambda$. The results are summarized in Table~\ref{table: Vary lambda}. It is observed that while the performance of the parametric algorithm remains almost independent of $\bar\lambda$, the optimality gap obtained by Gurobi remains large, except for the extreme values of $\bar\lambda$ that correspond to nearly fully dense or fully sparse solutions.

\begin{table}[h!]
	\begin{center}
		\caption{Performance comparison for varying regularization}
		\label{table: Vary lambda}
		\begin{tabular}{ c|c|c|c|c|c|c|c } 
			\hline
			\rule{0pt}{10pt}\multirow{2}{*}{$\mathbf{Metric}$}&\multirow{2}{*}{$\mathbf{Method}$} & $\pmb{\bar\lambda=0.25}$& $\pmb{\bar\lambda=2.5}$& $\pmb{\bar\lambda=7.5}$& $\pmb{\bar\lambda=12.5}$&$\pmb{\bar\lambda=25}$&$\pmb{\bar\lambda=50}$\\
			&&NZ $\approx91\%$ &NZ $\approx$ 72\%&NZ $\approx$ 50\%&NZ $\approx$ 36\%&NZ $\approx 8\%$\ &NZ $=0\%$\  \\ 
			\hline
			\multirow{2}{*}{Time(s)} & Parametric &1.14 &1.07 &1.01 &0.97 &0.93 & 0.95\\ 
			& Gurobi &21.04 & TL& TL& TL& TL& 3.32 \\ 
			\hline
			B\&B nodes & \multirow{2}{*}{Gurobi}&5746&5667949&6186925&6048614&8477823&1.8
			\\
			Opt. gap & &0.01\%&0.17\%&1.48\%& 6.74\%& 60.05\%&0.00\%\\
			\hline
		\end{tabular}
	\end{center}
	\footnotesize TL: Time Limit (1 hour), NZ: percentage of non-zero elements in the optimal solution $x^\star$. The reported results are averaged over 5 trials.
\end{table}

Finally, we focus on the special case of path graphs. Specifically, we compare our parametric algorithm (Algorithm~\ref{alg:path}) to the direct DP approach proposed in~\cite{liu2023graph}. As discussed in Section~\ref{sec: drawback of DP}, the direct DP approach solves instances with path structure in $\mathcal{O}(n^2)$ time complexity. While this runtime matches the theoretical complexity of our parametric algorithm, Figure~\ref{fig: Exp1_tri_diag} illustrates that their practical performance differs. In particular, while the direct DP approach outperforms the parametric algorithm for $n\leq 2,000$, its runtime scales almost quadratically with $n$. On the other hand, the practical performance of the parametric algorithm scales almost linearly with $n$, enabling it to outperform the direct DP approach for larger instances $n> 2,000$.

\begin{figure}[h!]
	\begin{center}
		\includegraphics[scale=0.5]{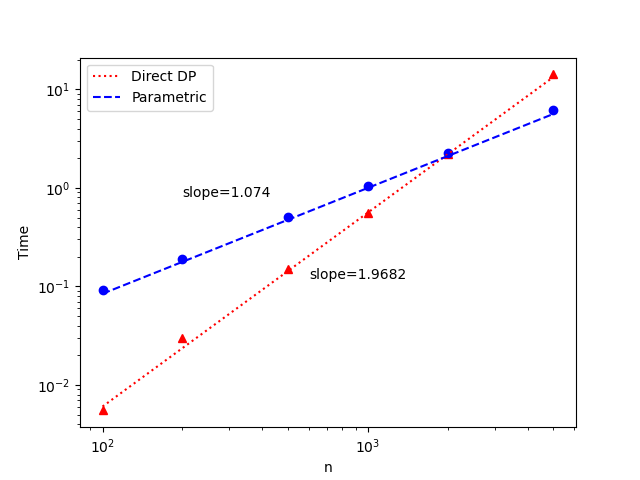}
	\end{center}
	\caption{The runtime of the parametric algorithm (Algorithm~\ref{alg:path}) and the direct DP approach of~\cite{liu2023graph} for instances with path structure. The reported results are averaged over 5 trials.}
	\label{fig: Exp1_tri_diag}
\end{figure}
	
	\subsection{Case Study on accelerometer dataset}\label{sec: Experiments HHM}
	In this case study, we highlight the performance of the parametric algorithm for solving the robust inference of GHMM, as detailed in Section~\ref{sec: HMM}. Specifically, we focus on the task of recognizing physical activities for a participant using data collected from a single chest-mounted accelerometer. We consider the dataset from \cite{dataset1_casale2012personalization,dataset2_casale2011human}. To enhance the representation of these activities, \cite{atamturk2021sparse} proposed using the mean absolute value of 10 successive signal differences from this dataset. The pre-processed data can be accessed online at \url{https://sites.google.com/usc.edu/gomez/data}. 
 
 We utilize the same dataset in our study.
The signal comprises 13,800 readings indicating changes in ``x acceleration'' for a participant. The participant's activity sequence is as follows: they were ``working at a computer" until timestamp 4,415; then engaged in ``standing up, walking, and going upstairs" until timestamp 4,735; followed by ``standing" from timestamp 4,735 to 5,854, from 8,072 to 9,044, and again from 9,045 to 9,720. Subsequently, they were ``walking" from timestamp 5,854 to 8,072; involved in ``going up or down stairs" from timestamp 9,044 to 9,435; ``walking and talking with someone" from timestamp 9,720 to 10,430; and ``talking while standing" from timestamp 10,457 to 13,800 (with the status between timestamps 10,430 and 10,457 being unknown).
	 
	 This problem can be formulated as an instance of Problem~\eqref{eq: HMM}, where the hidden state $x_t$ represents the activity level of the participant. Specifically, intervals characterized by minimal or absent physical activity naturally correspond to time stamps $t$ where $x_t=0$. Furthermore, we segment the signal into windows of magnitude $K$ and regard each segment $t$ as the observation set for the hidden state $x_t$. More precisely, we treat ${y_{(t-1)K+1}, \dots, y_{tK}}$ as the observations corresponding to the hidden state $x_t$. 

  \begin{figure}
    \centering
    \includegraphics[scale=0.39]{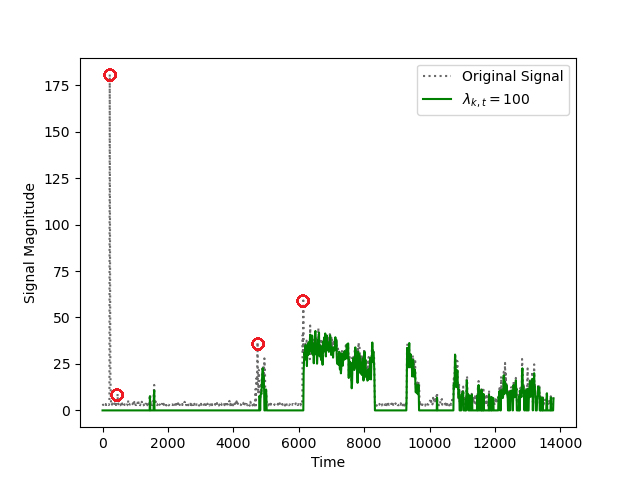}\includegraphics[scale=0.39]{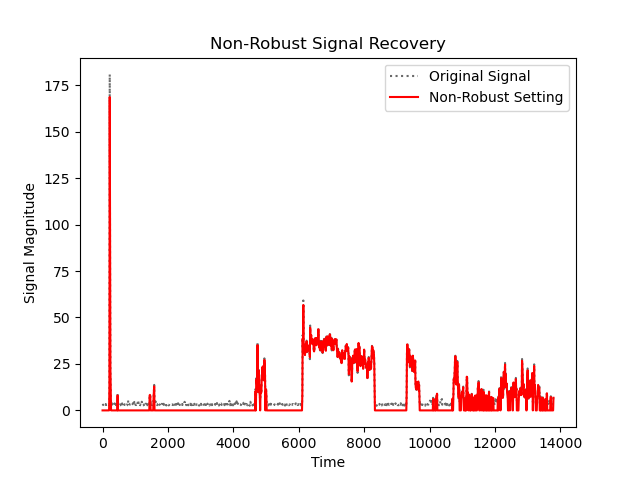}
    \caption{Robust and non-robust inference of the hidden signal. In the figure on the left, the outliers removed from the signal are circled in red. The parameters in this experiment are set to $\gamma_t=400,\lambda_{k,t}=100,\sigma_t^2=2$, and $\nu_t=1$.}
    \label{fig: Robust vs non-robust}
\end{figure}
  
Additionally, we assume that a subset of the observations is corrupted with outlier noise. As discussed in Section~\ref{sec: HMM}, the inference of a GHMM with outliers (referred to as robust inference hereafter) can be addressed by solving Problem~\eqref{eq: HMM}. Since this problem has a tree structure, it can be solved via the parametric algorithm. In this context, the scale of the problems being addressed exceeds $n=30,000$. At such scales, Gurobi fails to yield a reliable solution. Alternatively, in scenarios where the observations are assumed to be free of outliers, the variables $w$ and $z$ in Problem~\eqref{eq: HMM} can be set to zero. This transformation simplifies the problem into one defined over a path graph, which can be solved using the parametric algorithm over path graphs (Algorithm~\ref{alg:path}) or the direct DP approach proposed in~\cite{liu2023graph}.

 \begin{figure}
\begin{center}
	\includegraphics[scale=0.3]{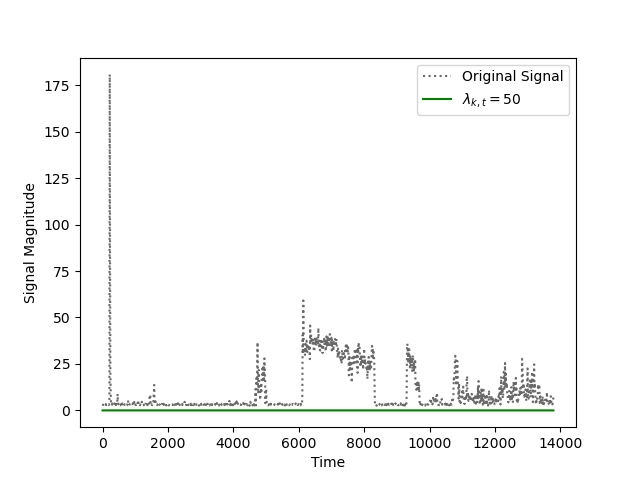}\includegraphics[scale=0.3]{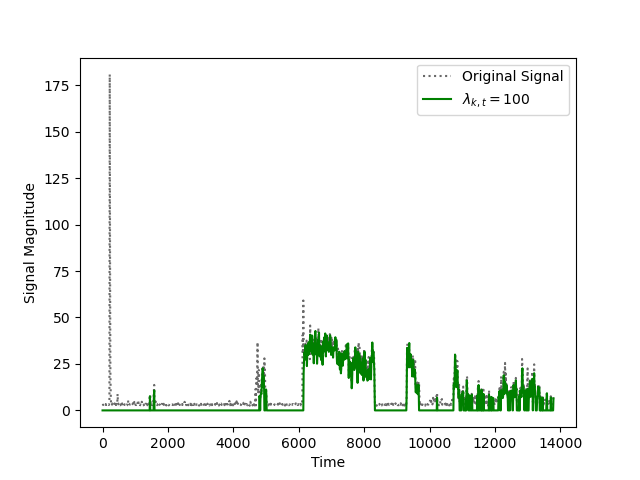}\\
	\includegraphics[scale=0.3]{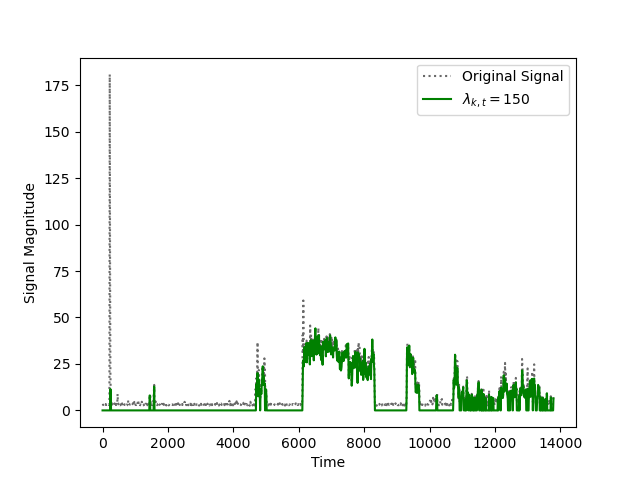}\includegraphics[scale=0.3]{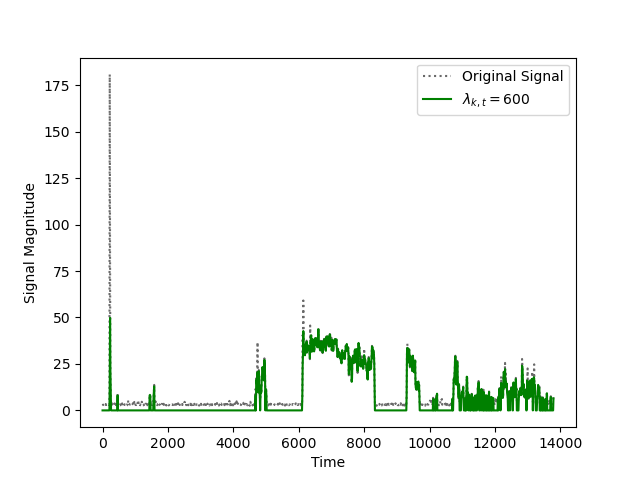}
    \caption{The recovered signal for $\lambda_{k,t}\in\{50,100,150,600\}$. The other parameters are set to $\gamma_t=400,\sigma_t^2=2, \nu_t=1$, and $K=10$.}
    \label{fig: Exp2_vary_w}
\end{center}
\end{figure}

Figure~\ref{fig: Robust vs non-robust} depicts the robust and non-robust inference of the hidden signal for $K=10$. It is evident that the original signal is corrupted with outlier noise, with the most significant outlier appearing at timestamp 250. While the robustly recovered signal successfully filters out the outliers, its non-robust counterpart fails to remove them.  In these experiments, our parametric algorithm solves the robust inference problem within $46.4$ seconds, whereas the non-robust inference is solved within $1.2$ seconds. This disparity in runtimes is not surprising, given that the robust inference problem is nearly 11 times larger.
	
Figure~\ref{fig: Exp2_vary_w} depicts the impact of the regularization parameter $\lambda_{k,t}$ on the recovered signal. A small value of $\lambda_{k,t}$ results in a fully dense $w$, effectively treating the entire observations as corrupted by outlier noise. Conversely, a larger $\lambda_{k,t}$ enforces sparser $w$, indicating that most observations are assumed to be free of outlier noise.

\begin{figure}[h!]
	\begin{center}
		\includegraphics[scale=0.5]{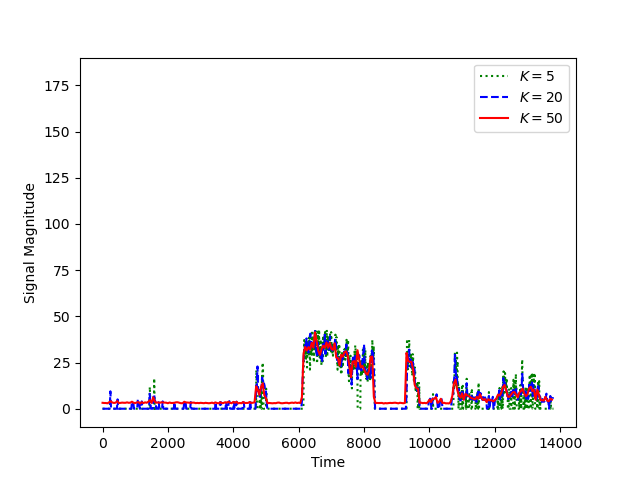}
	\end{center}
	\caption{The recovered signal for three values of $K$. The parameters are set to $\gamma_t=250,\lambda_{k,t}=100,\sigma_t^2=2$, and $\nu_t=1$.}
	\label{fig:K}
\end{figure}

 Finally, Figure~\ref{fig:K} illustrates the impact of varying values of the partition size $K$ on the recovered signal. Recall that $K$ represents the number of observations for each hidden state.  As a result, a larger $K$ typically improves the smoothness of the recovered signal but could potentially obscure finer changes. This phenomenon is shown in Figure~\ref{fig:K}.
\paragraph{Online Setting:}
Finally, we consider the online setting, where the goal is to infer the values of the hidden state $x_t$, as the new collected data from the accelerometer arrives ``on-the-go''. More specifically, at each new timestep $t=1,\dots, T$, $K$ new observations are revealed, and the goal is to infer the value of $x_t$, and possibly update the values of $S$ most recent values $x_{t-1}, \dots, x_{t-S}$ based on the newly observed data. 
Note that new observations at current time $t$ not only help with the inference of the current hidden state $x_t$, but also can potentially change the optimal value of the past hidden states $x_{t-1}, \dots, x_1$. Consequently, the optimal inference of the hidden state necessitates resolving a sequence of optimization problems with the new incoming data. 

Thanks to our parametric approach, we achieve this goal in \textit{milliseconds}. To see this, note that our parametric algorithm performs inference by sequentially obtaining $f_{x_1}, \dots, f_{x_{t-1}}$ corresponding to the parametric costs at the hidden states $x_1, \dots, x_{t-1}$, along with their conjugates (refer to Figure~\ref{fig: Hidden Markov Model} for an illustration of the associated graph). Therefore, according to the recursive equation~\eqref{eq:recursion}, the parametric cost $f_{x_t}$ at the new time $t$ can be efficiently characterized merely based on the conjugate functions $f^*_{x_{t-1}}$ (which is already computed and available) and $\{f^*_{y_{k,t}}\}_{k=1}^K$, thus circumventing the need to resolve Problem~\eqref{eq: HMM} from scratch. Once the parametric cost $f_{x_t}$ is obtained, the hidden states $x_t, \dots x_{t-S}$ can be updated in $\mathcal{O}(S)$, according to Algorithm~\ref{alg: general trees}. Figure~\ref{fig: Exp_Online} illustrates the runtime of this online version of our algorithm. At any given time $t$, the optimal cost, along with the updated values of $x_t, \dots, x_{t-4}$ are obtained based on $K=10$ new observations within at most 45 milliseconds.

\begin{figure}[h]
    \centering
    \includegraphics[scale=0.5]{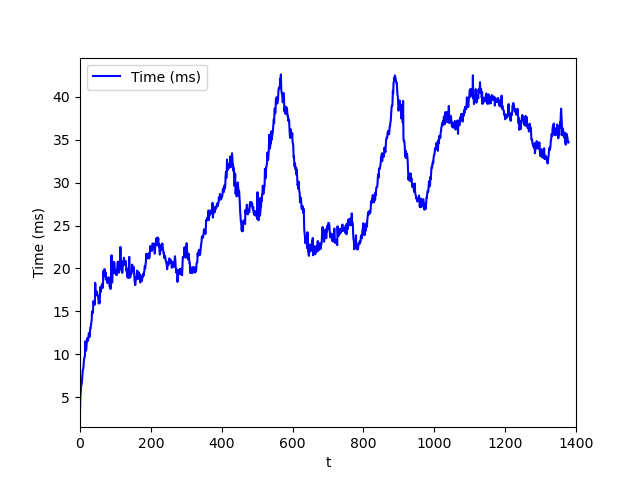}
    \caption{The update time of the 5 most recent hidden states after the arrival of $K=10$ observations. The other parameters are set to $\gamma_t=250,\lambda_{k,t}=100,\sigma^2_t=2,\nu_t=1$.}
    \label{fig: Exp_Online}
\end{figure}
 \section{Conclusions}
 In this paper, we consider mixed-integer quadratic programs with indicators where the Hessian of the quadratic term, $Q$, has a tree structure. While for general $Q$ the problem is NP-hard, we propose a highly efficient algorithm for the tree-structured $Q$. Our algorithm has a time and memory complexity of  $\mathcal{O}(n^2)$ that maintains the same complexity as the simpler path-structured problem studied earlier. Our computational results show that the practical complexity of the algorithm on our test instances is almost linear. 
Our algorithm can be leveraged in problems where the $Q$ matrix can be decomposed into trees in a similar procedure proposed in \cite{liu2023graph}.

	\bibliographystyle{splncs04}
	\bibliography{mybibliography.bib}

\begin{thebibliography}{10}
\providecommand{\url}[1]{\texttt{#1}}
\providecommand{\urlprefix}{URL }
\providecommand{\doi}[1]{https://doi.org/#1}

\bibitem{ahuja1988network}
Ahuja, R.K., Magnanti, T.L., Orlin, J.B.: Network flows  (1988)

\bibitem{atamturk2018strong}
Atamt{\"u}rk, A., G{\'o}mez, A.: Strong formulations for quadratic optimization
  with {M-matrices} and indicator variables. Mathematical Programming
  \textbf{170}(1),  141--176 (2018)

\bibitem{atamturk2021sparse}
Atamt{\"u}rk, A., G{\'o}mez, A., Han, S.: Sparse and smooth signal estimation:
  Convexification of $\ell_0$-formulations. The Journal of Machine Learning
  Research  \textbf{22}(1),  2370--2412 (2021)

\bibitem{bertsimas2023new}
Bertsimas, D., Cory-Wright, R., Pauphilet, J.: A new perspective on low-rank
  optimization. Mathematical Programming pp. 1--46 (2023)

\bibitem{bertsimas2016subset}
Bertsimas, D., King, A., Mazumder, R.: Best subset selection via a modern
  optimization lens. The Annals of Statistics pp. 813--852 (2016)

\bibitem{bertsimas2020sparse}
Bertsimas, D., Van~Parys, B.: Sparse high-dimensional regression: Exact
  scalable algorithms and phase transitions  (2020)

\bibitem{besag1974spatial}
Besag, J.: Spatial interaction and the statistical analysis of lattice systems.
  Journal of the Royal Statistical Society: Series B (Methodological)
  \textbf{36}(2),  192--225 (1974)

\bibitem{besag1995conditional}
Besag, J., Kooperberg, C.: On conditional and intrinsic autoregressions.
  Biometrika  \textbf{82}(4),  733--746 (1995)

\bibitem{brown1992introduction}
Brown, R.G., Hwang, P.Y.: Introduction to random signals and applied Kalman
  filtering, vol.~3. Wiley New York (1992)

\bibitem{dataset2_casale2011human}
Casale, P., Pujol, O., Radeva, P.: Human activity recognition from
  accelerometer data using a wearable device. In: Pattern Recognition and Image
  Analysis: 5th Iberian Conference, IbPRIA 2011, Las Palmas de Gran Canaria,
  Spain, June 8-10, 2011. Proceedings 5. pp. 289--296. Springer (2011)

\bibitem{dataset1_casale2012personalization}
Casale, P., Pujol, O., Radeva, P.: Personalization and user verification in
  wearable systems using biometric walking patterns. Personal and Ubiquitous
  Computing  \textbf{16},  563--580 (2012)

\bibitem{ceria1999convex}
Ceria, S., Soares, J.: Convex programming for disjunctive convex optimization.
  Mathematical Programming  \textbf{86},  595--614 (1999)

\bibitem{chakrabarty2017subquadratic}
Chakrabarty, D., Lee, Y.T., Sidford, A., Wong, S.C.w.: Subquadratic submodular
  function minimization. In: Proceedings of the 49th Annual ACM SIGACT
  Symposium on Theory of Computing. pp. 1220--1231 (2017)

\bibitem{chang1988estimation}
Chang, I., Tiao, G.C., Chen, C.: Estimation of time series parameters in the
  presence of outliers. Technometrics  \textbf{30}(2),  193--204 (1988)

\bibitem{chen2014NPhard}
Chen, X., Ge, D., Wang, Z., Ye, Y.: Complexity of unconstrained minimization.
  Mathematical Programming  \textbf{143}(1-2),  371--383 (2014)

\bibitem{das2008algorithms}
Das, A., Kempe, D.: Algorithms for subset selection in linear regression. In:
  Proceedings of the Fortieth Annual ACM Symposium on Theory of Computing. pp.
  45--54 (2008)

\bibitem{dedieu2021learning}
Dedieu, A., Hazimeh, H., Mazumder, R.: Learning sparse classifiers: Continuous
  and mixed integer optimization perspectives. The Journal of Machine Learning
  Research  \textbf{22}(1),  6008--6054 (2021)

\bibitem{del2020subset}
Del~Pia, A., Dey, S.S., Weismantel, R.: Subset selection in sparse matrices.
  SIAM Journal on Optimization  \textbf{30}(2),  1173--1190 (2020)

\bibitem{fattahi2021scalable}
Fattahi, S., Gomez, A.: Scalable inference of sparsely-changing gaussian markov
  random fields. Advances in Neural Information Processing Systems
  \textbf{34},  6529--6541 (2021)

\bibitem{fattahi2023solution}
Fattahi, S., G{\'o}mez, A.: Solution path of time-varying {Markov} random
  fields with discrete regularization  (2023)

\bibitem{gomez2021outlier}
G{\'o}mez, A.: Outlier detection in time series via mixed-integer conic
  quadratic optimization. SIAM Journal on Optimization  \textbf{31}(3),
  1897--1925 (2021)

\bibitem{gomez2023outlier}
G{\'o}mez, A., Neto, J.: Outlier detection in regression: conic quadratic
  formulations. arXiv preprint arXiv:2307.05975  (2023)

\bibitem{gomez2024note}
G{\'o}mez, A., Xie, W.: A note on quadratic constraints with indicator
  variables: Convex hull description and perspective relaxation. Operations
  Research Letters  \textbf{52},  107059 (2024)

\bibitem{gunluk2010perspective}
G{\"u}nl{\"u}k, O., Linderoth, J.: Perspective reformulations of mixed integer
  nonlinear programs with indicator variables. Mathematical programming
  \textbf{124},  183--205 (2010)

\bibitem{gunluk2011perspective}
G{\"u}nl{\"u}k, O., Linderoth, J.: Perspective reformulation and applications.
  In: Mixed Integer Nonlinear Programming, pp. 61--89. Springer (2011)

\bibitem{han2021compact}
Han, S., G{\'o}mez, A.: Compact extended formulations for low-rank functions
  with indicator variables. arXiv preprint arXiv:2110.14884  (2021)

\bibitem{han2022polynomial}
Han, S., G{\'o}mez, A., Pang, J.S.: On polynomial-time solvability of
  combinatorial {Markov} random fields  (2022)

\bibitem{hastie2017extended}
Hastie, T., Tibshirani, R., Tibshirani, R.J.: Extended comparisons of best
  subset selection, forward stepwise selection, and the lasso. arXiv preprint
  arXiv:1707.08692  (2017)

\bibitem{hazimeh2022sparse}
Hazimeh, H., Mazumder, R., Saab, A.: Sparse regression at scale:
  Branch-and-bound rooted in first-order optimization. Mathematical Programming
   \textbf{196}(1-2),  347--388 (2022)

\bibitem{huo2010NPhard}
Huo, X., Chen, J.: Complexity of penalized likelihood estimation. Journal of
  Statistical Computation and Simulation  \textbf{80}(7),  747--759 (2010)

\bibitem{insolia2022simultaneous}
Insolia, L., Kenney, A., Chiaromonte, F., Felici, G.: Simultaneous feature
  selection and outlier detection with optimality guarantees. Biometrics
  \textbf{78}(4),  1592--1603 (2022)

\bibitem{kalman1960new}
Kalman, R.E.: A new approach to linear filtering and prediction problems
  (1960)

\bibitem{kim2015hidden}
Kim, Y.J., Kang, B.N., Kim, D.: Hidden {Markov} model ensemble for activity
  recognition using tri-axis accelerometer. In: 2015 IEEE International
  Conference on Systems, Man, and Cybernetics. pp. 3036--3041. IEEE (2015)

\bibitem{kucukyavuz2022consistent}
K\"{u}\c{c}\"{u}kyavuz, S., Shojaie, A., Manzour, H., Wei, L., Wu, H.H.:
  Consistent second-order conic integer programming for learning {Bayesian}
  networks. Journal of Machine Learning Research  \textbf{24}(322),  1--38
  (2023)

\bibitem{lee2015faster}
Lee, Y.T., Sidford, A., Wong, S.C.w.: A faster cutting plane method and its
  implications for combinatorial and convex optimization. In: 2015 IEEE 56th
  Annual Symposium on Foundations of Computer Science. pp. 1049--1065. IEEE
  (2015)

\bibitem{liu2023graph}
Liu, P., Fattahi, S., G{\'o}mez, A., K{\"u}{\c{c}}{\"u}kyavuz, S.: A
  graph-based decomposition method for convex quadratic optimization with
  indicators. Mathematical Programming  \textbf{200}(2),  669--701 (2023)

\bibitem{Manzour21}
Manzour, H., K\"{u}\c{c}\"{u}kyavuz, S., Wu, H.H., Shojaie, A.: Integer
  programming for learning directed acyclic graphs from continuous data.
  INFORMS Journal on Optimization  \textbf{3}(1),  46--73 (2021)

\bibitem{orlin2009faster}
Orlin, J.B.: A faster strongly polynomial time algorithm for submodular
  function minimization. Mathematical Programming  \textbf{118}(2),  237--251
  (2009)

\bibitem{visweswaran2023efficient}
Ravikumar, V., Xu, T., Al-Holou, W.N., Fattahi, S., Rao, A.: Efficient
  inference of spatially-varying {Gaussian Markov} random fields with
  applications in gene regulatory networks. IEEE/ACM Transactions on
  Computational Biology and Bioinformatics  (2023)

\bibitem{stubbs1996branch}
Stubbs, R.A.: Branch-and-cut methods for mixed 0-1 convex programming.
  Northwestern University (1996)

\bibitem{trabelsi2013unsupervised}
Trabelsi, D., Mohammed, S., Chamroukhi, F., Oukhellou, L., Amirat, Y.: An
  unsupervised approach for automatic activity recognition based on hidden
  {Markov} model regression. IEEE Transactions on automation science and
  engineering  \textbf{10}(3),  829--835 (2013)

\bibitem{tsay1986time}
Tsay, R.S.: Time series model specification in the presence of outliers.
  Journal of the American Statistical Association  \textbf{81}(393),  132--141
  (1986)

\bibitem{Wei2024}
Wei, L., Atamt\"urk, A., G\'omez, A., K\"u\c{c}\"ukyavuz, S.: On the convex
  hull of convex quadratic optimization problems with indicators. Mathematical
  Programming  \textbf{204}(1-2),  703--737 (2024)

\bibitem{wei2021ideal}
Wei, L., G\'omez, A., K\"u\c{c}\"ukyavuz, S.: Ideal formulations for
  constrained convex optimization problems with indicator variables.
  Mathematical Programming  \textbf{192}(1-2),  57--88 (2022)

\bibitem{wei2020convexification}
Wei, L., G{\'o}mez, A., K{\"u}{\c{c}}{\"u}kyavuz, S.: On the convexification of
  constrained quadratic optimization problems with indicator variables. In:
  International Conference on Integer Programming and Combinatorial
  Optimization. pp. 433--447. Springer (2020)

\bibitem{xie2020scalable}
Xie, W., Deng, X.: Scalable algorithms for the sparse ridge regression. SIAM
  Journal on Optimization  \textbf{30},  3359--3386 (2020)

\bibitem{yan2022real}
Yan, H., Grasso, M., Paynabar, K., Colosimo, B.M.: Real-time detection of
  clustered events in video-imaging data with applications to additive
  manufacturing. IISE Transactions  \textbf{54}(5),  464--480 (2022)

\end{thebibliography}

\appendix

\section{Proof of Lemma~\ref{lem:f_quad}}
Let $J$ be the set of nodes in $\supp_u(Q)$, excluding $u$. Let us define 
$p_{u,s}(\alpha)$:
\begin{subequations}\nonumber
		\begin{align}
			p_{u,s}(\alpha) = \min_{x\in\bbbr^{n_u-1},z\in\{0,1\}^{n_u-1}}&\ \frac{1}{2}\alpha^2+c_u\alpha+\left(\frac{1}{2}x^\top Q_{J,J}x+ \alpha Q_{u,J}^\top x + c_J^\top x+\lambda_J^\top z\right)\\ 
			\text{s.t.}& \ x_i(1-z_i)=0\quad  i=1,2\ldots n_u-1\\
			& z=s.&
		\end{align}
	\end{subequations}
 It is easy to verify that $f_u(\alpha) = \min_{s\in\{0,1\}^{n_u-1}}\{p_{u,s}(\alpha)\}+\lambda_u\bbbone_\alpha$. Therefore, it remains to show that for every $s\in\{0,1\}^{n_u-1}$, $p_{u,s}(\alpha)$ is strongly convex and quadratic. To establish this, let $\mathcal{G}_s$ denote $\supp_u(Q)$ with nodes corresponding to $s_i=0$ removed. Consider $\{\mathcal{G}_s^g\}_{g=1}^{G_s}$, the connected components of $\mathcal{G}_s$, where $\mathcal{G}_s^1$ contains node $u$. Additionally, define $J_s^g$ as the node set within $\mathcal{G}_s^g$. It is evident that the above optimization decomposes into $G_s$ sub-problems defined over its connected components, with only one depending on $\alpha$. In particular, we have 
 $$p_{u,s}(\alpha) = p_{u,s,1}(\alpha)+\sum_{g=2}^{G_s}p_{u,s,g},$$
 where 
 \begin{align}
     p_{u,s,1}(\alpha) &= \min_{x}\left\{\frac{1}{2}\alpha^2+c_u\alpha+\frac{1}{2}x^\top Q_{J_s^1,J_s^1}x+ \alpha Q_{u,J_s^1}^\top x + c_{J_s^1}^\top x+\sum_{i\in J_s^1}\lambda_i\right\}\nonumber\\
     p_{u,s,g} &= \min_{x}\left\{\frac{1}{2}x^\top Q_{J_s^g,J_s^g}x+ c_{J_s^g}^\top x+\sum_{i\in J_s^g}\lambda_i\right\}, g=2,\ldots,G_s.\nonumber
 \end{align}
Using Karush-Kuhn-Tucker (KKT) conditions, one can verify that $p_{u,s,1}(\alpha)$ takes the following closed-form expression:
\begin{align*}
    p_{u,s,1}(\alpha) =& \frac{1}{2}\left(1-Q_{u,J_s^1}^\top\left(Q_{J_s^1,J_s^1}\right)^{-1}Q_{u,J_s^1}\right)\alpha^2\\
    &+\left(c_u-Q_{u,J_s^1}^\top\left(Q_{J_s^1,J_s^1}\right)^{-1}c_{J_s^1}\right)\alpha+\left(-\frac{1}{2}c_{J_s^1}^\top\left(Q_{J_s^1,J_s^1}\right)^{-1}c_{J_s^1}+\sum_{i\in J_s^1}\lambda_i\right).
\end{align*}
Note that $\left(1-Q_{u,J_s^1}^\top\left(Q_{J_s^1,J_s^1}\right)^{-1}Q_{u,J_s^1}\right)$ is the Schur complement of $Q_{J_s^1\cup\{u\}, J_s^1\cup\{u\}}$, which, owing to the positive definiteness of $Q$, is positive definite. Therefore, $\left(1-Q_{u,J_s^1}^\top\left(Q_{J_s^1,J_s^1}\right)^{-1}Q_{u,J_s^1}\right)>0$, implying that $p_{u,s,1}(\alpha)$ is strongly convex. This completes the proof.\qed

\end{document}